\documentclass[11pt]{amsart}

\usepackage{amsfonts,amsmath,amssymb,amscd}    
\usepackage[colorlinks=true,linkcolor=blue,citecolor=blue,urlcolor=blue]{hyperref}
\usepackage{fullpage}
\title{Some results on contraction rates for Bayesian inverse problems}

\author{Madhuresh} \address{TIFR Centre for Applicable Mathematics,
  Post Bag No. 6503, GKVK Post Office, Bangalore-560065, India}
\email{madhuresh@math.tifrbng.res.in}



\thanks{This research work benefited from the support of the AIRBUS
  Group Corporate Foundation Chair in Mathematics of Complex Systems
  established in CAM and ICTS, TIFR, Bangalore. The authors would like
  to thanks Andrew Stuart for the initial discussions which motivated
  this work.}

\newtheorem{theorem}{Theorem}[section]

\newtheorem{assumption}[theorem]{Assumption}

\newtheorem{claim}[theorem]{Claim}

\newtheorem{corollary}[theorem]{Corollary}

\newtheorem{definition}[theorem]{Definition}
\newtheorem{example}[theorem]{Example}

\newtheorem{lemma}[theorem]{Lemma}

\newtheorem{proposition}[theorem]{Proposition}
\newtheorem{remark}[theorem]{Remark}

\newcommand{\real}{\mathbb{R}}

\newcommand{\PP}{\mathbb{P}}

\newcommand{\I}{{\mathcal I}}
\newcommand{\cov}{{\textrm{cov}}}
\newcommand{\calC}{{\mathcal C}}
\newcommand{\calF}{{\mathcal F}}
\newcommand{\calH}{{\mathcal H}}

\newcommand{\normal}{{\mathcal N}}

\def\eps{\epsilon}

\newcommand{\bbq}{{\mathbb Q}}
\newcommand{\bbw}{{\mathbb W}}
\newcommand{\bbqun}{{\mathbb Q}_{u,n}}
\newcommand{\htwohat}{{\widehat{\mathcal H}_2}}
\newcommand{\phiyu}{\Phi(y,u)}
\newcommand{\bbqon}{{\mathbb Q}_{u_0,n}}

\begin{document}
\begin{abstract}
We prove a general lemma for deriving contraction rates for linear inverse problems with 
non parametric nonconjugate priors. We then apply it to get contraction rates for both mildly 
and severely ill posed linear inverse problems with Gaussian priors in non conjugate cases. In the 
severely illposed case, our rates match the minimax rates using scalable priors with scales which 
do not depend upon the smoothness of true solution. In the mildly illposed case, our rates match 
the minimax rates using scalable priors when the true solution is not too smooth.

Further, using the lemma, we find contraction rates for inversion of a semilinear operator with 
Gaussian priors. We find the contraction rates for a compactly supported prior. We also discuss 
the minimax rates applicable to our examples when the Sobolev balls in which the 
true solution lies, are different from the usual Sobolev balls defined via the basis of forward operator.

\end{abstract}
\maketitle

\section{Introduction} \label{sec-intro}

Inverse problems can always be formulated as solutions of equations of
the type $y = G(u)$ where $y$ is known (typically some measurement or
data from an experiment) and one is interested in solving the equation to find $u$
\cite[Chapter~1]{Kirsch}. The model which we shall investigate here is
the statistical version of infinite dimensional linear inverse
problems of this type. Specifically, we assume that $G$ is a linear,
injective operator from Hilbert space $\calH_1$ to Hilbert space
$\calH_2$ (both infinite dimensional). We treat $u$ as a random
variable on $\calH_1$ and define $y$ by the following equation:
\begin{equation} \label{eqn:begin}
  y = G(u) + \frac{1}{\sqrt{n}}\eta \,,
\end{equation}
where $\eta$ is a Gaussian noise on $\calH_2$: $\eta \sim
\normal(0,\zeta)$.\footnote{ We may choose $\eta$ to be the white
  noise i.e.  $\zeta = \I$. This makes the data space large, because
  the white noise is supported on a much bigger space than
  $\calH_2$. We shall explain white noise with reasonable details in Section
\ref{sec-consistency}. }
In a Bayesian
formulation of the problem, also known as nonparametric (parametric)
Bayesian framework when $\calH_1$ is infinite (finite) dimensional, we
assume that $u$ is distributed according to a measure $\mu_n$ (we assume
that the choice of prior may depend on level of noise) on
$\calH_1$, to be considered as the prior measure. We further assume
that $u$ and $\eta$ are independent. Given the above assumptions,
Bayes' theorem gives the distribution of the conditional random
variable $u$ given $y$, called the posterior distribution, to be
denoted henceforth by $\mu^y_n$. We are interested in the
properties of the posterior in the small-noise limit as $n \to
\infty$.

In many practical applications, the inverse problem is ill-posed, in
the sense that $u$ may not exist for given $y$, or it may not be
unique, or $u$ may not depend continuously on $y$. In such cases, many
methods of regularization are well developed for linear ill-posed
inverse problems but are still being developed for many nonlinear
problems, see, e.g. the book by Kirsch \cite{Kirsch} and references
therein. A related approach in dealing with ill-posed inverse problems
is the statistical approach similar to the one introduced in the first
paragraph. The statistical approach may turn a deterministic ill-posed
inverse problem into a well-posed problem in the sense that, under
certain assumptions on $G$, $\mu_n$ and $\eta$, the posterior distribution 
$\mu_n^y$ for the conditional random variable $u$ given $y$ is continuous with
respect to $y$ even when $G^{-1}$ is not continuous (see for e.g., the
papers \cite{Stuart, Andrew} for details, further discussions and
references).

Thus the main object of interest in the statistical version is the
posterior distribution $\mu^y_n$.  Depending on the context, various
properties and quantities related to the posterior distribution are of
interest. A far from exhaustive list of recent studies includes
credible sets of the posterior \cite{Zanten}; computational methods
and finite dimensional approximations of equation~\eqref{eqn:begin}
and of the posterior $\mu^y_n$ \cite{Kaipio}; convergence of MAP 
(maximum a posteriori) and CM (conditional mean) estimators. For an
extensive bibliography, one may refer to \cite{Matti}.

Another major question in the statistical version is that of
consistency of the posterior measure. Depending on the formulation of
the problem, one expects the posterior measure to approach Dirac delta
supported on the true value of the unknown either as the number of
observations goes to infinity (large data limit), or as the
observations become more accurate (small noise limit). The first
result in large data limit is due to Doob \cite{locker-doob}, and this
limit has been investigated in many different contexts since
then. There has been a recent interest in the small noise limit
\cite{ray, Zanten, Zhang, Knapik}. In the case when $G$ is identity operator, 
then the asymptotics of large data and small noise limits are known to be equivalent 
\cite{brown, cavalier2008}. Consistency can be proved for all linear, injective $G$ 
when the prior is Gaussian using elementary methods, however we have not been 
able to locate a proof in literature. Methods
from the former (large data limit) are exploited in the work of Ray
\cite{ray} to deal with small noise limits and will be used explicitly
by us in this work as well.

We shall focus on the small noise limit. In particular, we will assume
that the data is obtained from a ``true'' unknown element $u_0 \in
\calH_1$ in equation~\eqref{eqn:begin}. We will be concerned with the
rate of contraction of the Bayesian posterior distribution to the true
solution $u_0$ in an appropriate way to be described
later. Essentially we look for the posterior measure $\mu^y_n$ of the
complements of small balls around $u_0$, that is, $\mu^y_n\{u: \|u -
u_0\|\ > \xi_n\}$. Note that this is a random variable in $y$. We then
study the convergence of this random variable to $0$ as $n \rightarrow
\infty$. In case we are able to prove such a convergence, $\xi_n$ is called
a {\bf contraction rate}. 


The novelty of this work, compared to the previous studies \cite{ray, Zanten, Zhang, Knapik, agapiou}, is to 
extend the class of models, mainly priors, but also
the class of distributions of noise $\eta$ for which we can
obtain contraction rates.
The details of the technical assumptions
needed for our results are given in Section~\ref{sec:main-lem}. 
We also give necessary and sufficient conditions for well-posedness of linear Bayesian models in Banach spaces (That is, model \ref{eqn:begin} with $\mathcal{H}_i$ being Banach spaces).  

In Section~\ref{subsec-setup}, we shall begin by presenting the basic
set-up, followed in Section~\ref{subsec-smallnoise} by the precise
definitions for consistency and contraction rates. 
We discuss our main contributions and relations to previous work in detail in
Section~\ref{subsec-discuss}. Our main
results on contraction rates, Lemma~\ref{lem:contraction}, as well as
Lemma~\ref{corollary1} and \ref{lemma r infinite}, which are
extensions of, and use the methods from the work of Ray \cite{ray}, are
stated and proved in Section~\ref{sec:main-lem}. Finally, in
Section~\ref{sec-examples}, we discuss different classes of examples, 
one of which is {\it semilinear}, where we can apply our results.

\section{Consistency and contraction rates} \label{sec-consistency}


In this section, we first introduce the basic setup and the Bayes'
theorem in the context of our problem, followed by the definitions of
consistency and contraction rates Section~\ref{subsec-smallnoise}. 
Heuristically, consistency implies that the random posterior measure
concentrates around the true solution as the observations get more
precise by way of the observational noise converging to zero in an
appropriate way. In a similar fashion, Contraction rates measure how quickly 
the above said concentration happens.


\subsection{Basic setup} \label{subsec-setup}

As discussed in Section~\ref{sec-intro}, we will
consider statistical versions of linear inverse problem, as given in
model \eqref{eqn:begin}. We will treat $u$ and $y$ as
random variables defined on an abstract probability space $(\Omega,
\calF, \PP)$, taking values in Hilbert spaces $\calH_1$ and $\calH_2$,
equipped with the inner products $\langle\cdot,\cdot\rangle_1$, and
$\langle\cdot,\cdot\rangle_2$ and corresponding norms $\|\cdot\|_1$
and $\|\cdot\|_2$, respectively. Here recall that $y$ is referred to as the
observed data and $u$ as the unknown, or the parameter to be
determined.

We will consider the operator $G: \calH_1 \to \calH_2$ to be linear
and injective. Our first lemma will require that $G$ have a singular
value decomposition, i.e., the eigenvectors of $G^TG$ form a basis of
$\calH_1$ which we denote by $\{e_k\}$ with corresponding eigenvalues
$\{\rho_k^2\}$, i.e., $G^T G e_k = \rho_k^2 e_k$ for $k =
1,2,\ldots$. We will consider ill-posed inverse problems in the sense
that the inverse of $G$ is discontinuous, i.e., $0$ is a limit point
of the eigenvalues $\rho_k^2$. Thus in this case, for any $y \in
G(\calH_1)$, there is a unique $u \in \calH_1$ satisfying $y = G(u)$,
but the dependence of $u$ on $y$ is not continuous. Also note that $y$ may
not even belong to $\calH_2$ or
$G(\calH_1)$ due to the white noise component. Two important cases of
ill-posed inverse problems are as follows \cite{ray,cavalier2008}.
\begin{itemize}
\item The inverse problem is called severely ill-posed if
  \begin{equation}\label{S}
    C_1(1 + k^2)^{-\alpha_1}e^{-C_0 k^{-\beta}} \leq \rho_k \leq
    C_2(1 + k^2)^{-\alpha_2}e^{-C_0 k^{-\beta}} \quad \textrm{as}
    \quad k \to \infty.
    \end{equation}
for some constants $\alpha_1,\alpha_2 \in \real$ and $C_0,C_1,C_2,\beta > 0$.
Essentially in this case, the eigenvalues of $G^TG$ decay
exponentially. An example is the classical inverse problem of
determination of initial condition of the heat equation given the
solution at a fixed later time.\\
\item The problem is called mildly ill-posed if
  \begin{equation}\label{M}
    C_1(1 + k^2)^{-\alpha/2} \leq \rho_k \leq C_2(1 +
    k^2)^{-\alpha/2} \quad \textrm{as} \quad k \to \infty.
  \end{equation}
for some constants $C_1,C_2,\alpha>0$.
\end{itemize}

We will consider the case of observational noise to be Gaussian, i.e.,
$\eta \sim \normal(0,\zeta)$, on $\calH_2$ and also assume that $G(\calH_1)$ is a subset of
$\calH_{\zeta}$, the Cameron-Martin space of the noise.

\begin{remark} Note that
$\normal(0,\zeta)$ is supported on $\calH_2$ if and only if $\zeta$ is
trace class, see \cite{Kuo}. However, covariance operators which are
not trace class are also commonly used. For instance, the covariance
of white noise is the identity which is clearly not trace class.  
In such a case, the noise is not actually supported on $\calH_2$ but on a much
larger space $\htwohat$ which we define below. Let $\{b_i\}$ be any
orthonormal basis of $\calH_2$ and let $R_i$ be the one dimensional
span of $b_i$. Let $\nu_i$ be the standard Gaussian $\normal(0,1)$ on
$R_i$ and define the white noise $\bbw$ on $\calH_2$ to be the product
measure of $\nu_i$ on $\htwohat := \prod R_i$. Even though the above
definition uses the basis $\{b_i\}$, it is easy to check\footnote{If $W$ is a random
variable distributed as $\bbw$, we can check that $\cov(\langle
W,x\rangle_2,\langle W,y\rangle_2) = \left\langle x,y\right\rangle_2$ (see \cite{Kuo}).}
that the measure $\bbw$ is independent
of the choice of the basis $\{b_i\}$.
\end{remark}

Denote by $\bbq_{0,n}$ the measure of the noise $\eta/\sqrt{n}$ and by
$\bbqun$ the measure of $y$ given $u$, i.e., measure of the scaled
noise $\eta/\sqrt{n}$ shifted by $G(u)$ for any fixed $u \in
\calH_1$. The above assumption that $G(u)$ belongs to the
Cameron-Martin space of $\bbq_{0,n}$ implies that $\bbqun \ll
\bbq_{0,n}$. Indeed, the Radon-Nikodym derivative is given by the
Cameron-Martin theorem \cite{Kuo}:
\begin{align}
  \frac{d\bbqun}{d\bbq_{0,n}} = \exp\left( - \phiyu \right) \,,
\label{eqn:rnderiv} \end{align}
where
\begin{equation}
  \phiyu := \frac{n}{2} \langle G(u), G(u) \rangle_{\zeta} - n \langle
  y, G(u) \rangle _{\zeta}\,, \label{eqn:phiyu}
\end{equation}
where $\langle x , y \rangle_{\zeta} = \langle \zeta^{-1/2}x,
\zeta^{-1/2}y \rangle_2$ is the Cameron-Martin inner product. Note that existence of $\phiyu$ 
assumes that the map $\zeta^{-1/2}G(u) \in \mathcal{H}_2$ for $u$ almost surely in prior measure.

Following the Bayesian philosophy, we will assume $u$ to be
$\calH_1$ valued random variable independent of $\eta$. The
distribution for $u$ will be interpreted as the \emph{prior} for $u$,
and will be denoted by $\mu_n$. We will be interested in the
properties of conditional distribution of $u$ given $y$, as given by the Bayes' theorem.

We will now describe the posterior measure and show that it is well defined. The following proposition is a generalization of Theorem 4.1 in \cite{Stuart} as it does not depend on any assumptions on the likelihood $\phiyu$ other than its existence.

\begin{claim} \label{zfinite1} Assume that the random variable $u$ is
  independent of the noise $\eta$. Further, assume that
  $\zeta^{-1/2}G(u) \in \calH_2$ for almost all $u$ with respect to the prior measure. 
  Then, the conditional
  distribution, denoted by $\mu_n^y$, for the random variable $u | y$
  is well-defined and also absolutely continuous with respect to the
  prior $\mu_n$ for $\bbq_{0,n}$-a.e.~$y$. Indeed, the posterior
  measure of any Borel set $B \subseteq \calH_1$ is given by
  \begin{align}
    \mu^y_n(B) = \frac{1}{Z^y_n} \int_B \exp\left( - \phiyu \right)
    d\mu_n(u) \,, \label{eqn:muyb}
  \end{align}
  where $\phiyu$ is defined in equation~\eqref{eqn:phiyu} and the
  constant
  \begin{align}
    Z^y_n := \int_{\calH_1} \exp\left( - \phiyu \right) d\mu_n(u)
  \end{align}
  is finite:
  \begin{align} \label{eq:Zfin}
    0 < Z^y_n < \infty \,.
  \end{align}
\end{claim}

\begin{proof}
Due to the independence of $u$ and $\eta$, the joint distribution of
  $(u,\eta)$ is given by $\mu_n \otimes \bbq_{0,n}$. Using
  equation~\eqref{eqn:begin} and the definition of $\bbqun$ above, we
  see that the distribution of the random variable $(u,y)$ is $\mu_n
  \otimes \bbqun$ which is absolutely continuous
  w.r.t.~$\mu_n \otimes \bbq_{0,n}$ with the same Radon-Nikodym
  derivative as the RHS of equation~\eqref{eqn:rnderiv}. Thus by
  Bayes' theorem, the conditional distribution of $u$ given $y$,
  denoted by $\mu^y_n$, is absolutely continuous with respect to
  $\mu_n$ with the same Radon-Nikodym derivative, which proves
  equation~\eqref{eqn:muyb}, as long as the constant $Z^y_n$ is finite
  but non-zero.

  We now show $0 < Z^y_n < \infty$ for $\bbq_{0,n}$-a.e.~$y$. It would
  suffice to show $0 < \int_SZ^y_n d\bbq_{0,n} < \infty$ for all sets
  of positive measure $S$ (under $\bbq_{0,n}$).
  \begin{eqnarray*}
    \int_SZ^y_n d\bbq_{0,n} &=& \int_S \left(\int_{\calH_1} \exp\left(
        - \phiyu \right) d\mu_n\right) d\bbq_{0,n}\\
    &=& \int_{\calH_1}\left(\int_S\exp\left( - \phiyu
      \right)d\bbq_{0,n}\right)d\mu_n\\
    &=& \int_{\calH_1} \bbq_{u,n}(S) d\mu_n\\
  \end{eqnarray*}
  We have used Tonelli theorem and the change of variable formula for
  translation of Gaussian measures in the above. Since $\bbq_{0,n}$
  and $\bbq_{u,n}$ are absolutely continuous with respect to each
  other for almost all $u$, we have $0 < \int \bbq_{u,n}(S) d\mu_n < \infty$
  and hence $0 < \int_SZ^y_n d\bbq_{0,n} < \infty$. Therefore the
  posterior distribution exists and has the density given by equation
  \eqref{eqn:muyb}.
\end{proof}

The above computation proves that the posterior $\mu^y_n$ is well
defined for $\bbq_{0,n}$ a.e. $y$. We also note here that the proof
solely relies on the existence of $\phiyu$, which is a consequence of
the assumption that $\zeta^{-1/2}G(u) \in \calH_2$ for almost all $u$ with 
respect to the prior measure.

\subsection{Definition of contraction rates} \label{subsec-smallnoise}

Posterior distribution, as seen above, can indeed be considered as a {\it solution} 
to a statistical inverse problem. However, such proposed solution needs to be tested
for some obvious flaws. Heuristically, assuming that the observation $y$ corresponds
a specific {\it true solution} $u_0\in\calH_1$, the posterior distribution should {\it concentrate} 
around the true solution $u_0$ as the intensity of noise decreases to zero.

The main idea is to estimate the posterior measure of complements of neighborhoods 
of the true solution, and show that they converge to zero. 

In particular, we will define 
$f^y_{n,\xi} := \mu_n^y \{\, u \mid \left\| u - u_0\right\| > \xi \,\}$ 
which is a random variable that is defined for each $y$ in
the support of $\bbq_{u_o,n}$. The consistency of the statistical
inverse problem is then defined by convergence in probability of
random variables $f^y_{n,\xi}$.

\begin{definition} \label{def-consist} The sequence of posteriors
  $\mu_n^y$ is said to be consistent if $\bbqon \{\, y \mid
  f^y_{n,\xi} > \delta \,\} \rightarrow 0$ as $n \rightarrow \infty$
  for every $\delta > 0$.
\end{definition}
In fact, the above convergence may be true even when we replace $\xi$
with a sequence $\xi_n \to 0$, which essentially defines the
contraction rate as follows.
\begin{definition}\label{def-contraction}
  A sequence $\xi_n \to 0$  is said to be a rate of contraction for the
  sequence of posterior measures $\mu_n^y$ if $\bbqon \{\, y \mid
  f^y_{n,\xi_n} > \delta \,\} \rightarrow 0$ as $n \rightarrow \infty$
  for every $\delta > 0$.
\end{definition}

 As mentioned earlier in the introduction, we prove a general lemma for
deriving contraction rates for linear inverse problems. We also compare our rates with minimax rates. 
A description of minimax rates can be found in \cite{Belister, cavalier2008}. Before we state and 
prove the lemma in Section~\ref{sec:main-lem}, we discuss relation of our results to previous work.

\subsection{Definition of well-posedness}\label{def-well-posedness}

The concept of well-posedness captures the fact that the posterior varies continuously with observation. 
In order to formalise the concept, we need appropriate metrics on the space of posteriors as well as observations. 
It is standard to use Hellinger distance as a metric on the space of posteriors. We use the definition given in \cite{Stuart}.
\begin{definition}
Given two probability measures $\mu$, $\mu'$ and a measure $\nu$ such that $\mu$ and $\mu'$ 
have densities with respect to $\nu$, the Hellinger distance $d_{H}(\mu,\mu')$ is defined as 
$$d_{H}(\mu,\mu') \equiv \sqrt{\left(\frac{1}{2}\int \left(\sqrt{\frac{d\mu}{d\nu}} - \sqrt{\frac{d\mu'}{d\nu}} \right)^2 d\nu\right)}$$
\end{definition}
\begin{definition}
The posterior model  $y \to \mu^y$ is said to be well posed if there exists a Banach space $(Y, \|.\|_Y)$ such that $y \in Y$ almost surely(note that $y$ is the sum of random variables $G(u)$ and $\frac{1}{\sqrt{n}}\eta$) and $\mu^y$ is a continuous function of $y$ with respect to the corresponding metrics
\end{definition}

\subsection{Our contribution and relation to previous work} \label{subsec-discuss}
Well-posedness of the posterior for inverse problems on Banach spaces has been discussed 
in \cite{Stuart} in a very general setting. It provides certain technical assumptions which are 
sufficient to show well-posedness. However, it can be cumbersome to show that the 
assumptions hold even in simple cases as can be seen in \cite{agapiou}. 

We give necessary and sufficient conditions for the posterior to be well-posed when the model \ref{eqn:begin} is defined on Banach spaces \ref{well-posedness}.

In the finite dimensional setup, a nondegenerate linear $G$ ensures well-posedness of the inverse problem and makes the problem of finding contraction rates easy. The same clearly does not hold for infinite dimensional case, which has received considerable attention only recently. We shall now outline
some of these studies of contraction rates in infinite dimensions,
pointing out the relations to our main result.

In the papers \cite{Zanten} and \cite{Zhang, Knapik}, the authors deal
with mildly ill-posed \eqref{M} and severely ill-posed
\eqref{S} inverse problems, respectively. Both these
studies use scalable Gaussian priors and white noise. They first calculate
the mean and covariance operator of the posterior distribution (which
is Gaussian as well) and find a bound for $f_{n,\xi}^y$ defined above,
using Markov inequality and then take the expectation of the resulting
random variable. At this point, they use the assumption that
$\varphi_k = e_k$ where $\{\varphi_k\}$ is the basis of $\calH_1$ with respect to which
the covariance of prior is
diagonalizable (recall that $\{e_k\}$ is the eigenbasis of $G^TG$). Using
this assumption, they show that the expectation of $f_{n,\xi}^y$
equals the sum of a series and the estimate for the value of the sum
provides the contraction rate. In \cite{Vaart}, the authors provide adaptive priors to get contraction rates which match the minimax rates (except for a logarithmic term) when using priors which are diagonalizable in the basis of $G^TG$.

In the paper \cite{agapiou}, the authors again work with Gaussian
priors, but where $\varphi_k = e_k$ may not hold. Further, observational noise may not be white, generalizing some recent results
\cite{florens-simoni2012, florens-simoni2014, Zanten, Knapik}. They
work with the explicit expressions for the mean and covariance of the
Gaussian posterior, obtaining contraction rates for mildly ill-posed
problems. The paper makes use of certain technical assumptions comparing the covariance operators of noise and prior measures to the linear operator $G$.
The paper gets contraction rates only when the true solution $u_0$ lies in the Cameron-martin space of the prior, in particular when $u_0 \in \calH^\gamma$ for $\gamma \ge 1$, 
where $\calH^\gamma$ is the Hilbert scale of order $\gamma$ defined in Section~3 of \cite{agapiou}.

Finally in \cite{ray}, the author discusses cases where the prior
may not be Gaussian and we shall be using methods used in that
paper. The main idea in \cite{ray} is to use test functions
as introduced in some earlier works \cite{Ghosal}. Similar test
functions are defined later in our paper in
Proposition~\ref{prop:test}. A key result, Theorem~2.1 in \cite{ray},
is along the lines of similar results in \cite{gine-nickl2011}, and is
proved under the assumption that 
\begin{equation} \label{condition-ray}
 |\{l \mid \langle \varphi_k ,
e_l \rangle_1 \ne 0 \}| < \infty
\end{equation}  for all $k$, where $\{\varphi_k\}$
is an arbitrary basis of $\calH_1$. However, when considering Gaussian
priors, the conditions required for the main Theorem~2.1 in \cite{ray} are
verified only for case when $\varphi_k = e_k$, which is also the case
discussed in \cite{Zanten, Zhang, Knapik, Vaart}. In conclusion, we would like to say that there is a paucity of results on contraction rates in non-conjugate (non-conjugate is used to mean that $\phi_i \neq e_i$) cases with Gaussian priors. The available contraction rates hold only when the true solution is in the Cameron-Martin space of the prior. Also, there are no results for severely ill posed problems in non conjugate cases with Gaussian priors. The main contribution of this paper is in the aspect which we summarise below.

\subsubsection{Our contribution}\label{contribution}

\begin{itemize}
\item[1)] {\bf Theorem \ref{well-posedness}:} We show that the posterior for model \eqref{eqn:begin} is well-posed for Gaussian priors if and only if $G(u)$ lies in the Cameron-Martin space of the noise almost surely with respect to the prior measure.
\item[2)] {\bf Lemma \ref{lemma r infinite}:} We weaken condition \eqref{condition-ray} of \cite{ray} by demanding only that 
\begin{equation}\label{condition-our}
  (G^{-1})^T\phi_k \in \mathcal{H}_2
\end{equation}  
for all $k$. 
\item[3)] {\bf Subsection \ref{sec:meyer}:} We then verify the conditions imposed by the lemma to examples of mildly illposed problems where $\varphi_k \neq e_k$ for scalable Gaussian priors and achieve minimax rates when true solution does not lie in the cameron martin space of the prior. The rates however, are suboptimal when the true solution lies in the Cameron-Martin space of the prior. Our examples strictly contain the class of examples discussed in the paper \cite{agapiou} (in the sense that we allow for a much larger class of perturbations) and we get contraction rates for true solutions belonging to all Hilbert scales/Sobolev classes (Subsection \ref{sec:G without SVD}). We also get contraction rates for the deconvolution problem using Gaussian priors on Meyer wavelet basis.
\item[4)] {\bf  Subsection \ref{severely ill posed}:} We obtain the minimax rates for severely ill posed problem in non conjugate cases using scalable priors (with scales independent of the smoothness of true solution) for all Sobolev 
classes of true solutions.
\item[5)] {\bf Subsection \ref{sec:semilinear}:} Finally, under appropriate assumptions, we also obtain the 
contraction rates for the posterior when the operator $G$ is a semilinear.
\end{itemize}


\section{Well-posedness and Contraction Lemmas}
\label{sec:main-lem}

In this section, we shall present our main results concerning wellposedness of the posterior, and
contraction rates for various choices of model parameters.

\subsection{Well-posedness}\label{well-posedness}

We shall begin with stating our result on wellposedness of the posterior on separable Hilbert spaces and then use it to show the result on Banach spaces.

\begin{theorem}\label{thm:well-posedness}
Consider the model \eqref{eqn:begin} with Gaussian prior $\mu \equiv N(0,\mathcal{C})$ and noise 
$\eta \sim \nu \equiv N(0,\zeta)$. Then, the posterior measure is well-posed 
and is locally Lipshitz in the observation with respect to the appropriate norm,
if and only if $G(u)$ lies almost surely in the Cameron-Martin space of $\nu$.

\end{theorem}

\begin{proof}
Following the proof of Proposition 3.1 in \cite{Zanten}, we write the posterior explicitly as a 
Gaussian measure $N(Ay,Q)$ with 
\begin{eqnarray*}
A &=& \mathcal{C}G^T\left(\zeta + G\mathcal{C}G^T\right)^{-1}\\
&=& \mathcal{C}G^T\zeta^{-\frac{1}{2}}\left(I + \zeta^{-\frac{1}{2}}G\mathcal{C}G^T\zeta^{-\frac{1}{2}}\right)^{-1}\zeta^{-\frac{1}{2}}
\end{eqnarray*}
and

\begin{eqnarray*}
Q &=& \mathcal{C} - A\left(\zeta + G\mathcal{C}G^T\right)^{-1}A^T\\
&=& \mathcal{C} - \mathcal{C}G^T\left(\zeta + G\mathcal{C}G^T\right)^{-1}G\mathcal{C}\\
&=& \mathcal{C}\left( I - \left(G^{-1}\zeta(G^{T})^{-1} + \mathcal{C}\right)^{-1}\mathcal{C}\right)\\
&=& \mathcal{C}\left( I + G^{T}\zeta G\mathcal{C}\right)^{-1}.\\
\end{eqnarray*}

Note that $Q$ is independent of the observation $y$. Using Fernique's theorem (\cite{Kuo}) 
and bounded convergence theorem, it is easy to see that $d_H(\mu^{y}, \mu^{y'})$ is locally 
Lipshitz with respect to $\|A(y)\|_Q = \|Q^{-\frac{1}{2}}A(y)\|_1$ where $\|\cdot\|_1$ is the usual 
norm on $\mathcal{H}_1$. 
If $G(u)$ lies in the Cameron-Martin space of noise almost surely with respect to the prior, then we have 
$$\int\|G(u)\|_{\zeta}^2 d\mu= \int \|\zeta^{-\frac{1}{2}}G(u)\|^2 d\mu < \infty,$$
where the norm $\|z\|^2_{\zeta} = \|\zeta^{-\frac{1}{2}}\,z\|^2$ is the Cameron-Martin norm of noise.
This implies that the covariance of the pushforward of $\mu$ under $\zeta^{-\frac{1}{2}}G$ is trace class. This implies that $\zeta^{-\frac{1}{2}}G\mathcal{C}G^T\zeta^{-\frac{1}{2}}$ is trace class.
Similar calculations show that if $Ay$ lies in the Cameron-Martin space of the posterior almost surely with respect to the distribution of $y$ (distributed as $G(u) + \eta$, that is, Gaussian with mean $0$ and covariance $\zeta + G\mathcal{C}G^T$), then $Q^{-\frac{1}{2}}A(\zeta + G\mathcal{C}G^T)A^T Q^{-\frac{1}{2}}$ is trace class. Since both the summands are positive, both $Q^{-\frac{1}{2}}A\zeta A^TQ^{-\frac{1}{2}}$ and $Q^{-\frac{1}{2}}AG\mathcal{C}G^TA^TQ^{-\frac{1}{2}}$ need to be trace class. With these facts in hand, we prove the theorem.

\textbf{The if part}

We assume that $\zeta^{-\frac{1}{2}}G\mathcal{C}G^T\zeta^{-\frac{1}{2}}$ is trace class. Well-posedness of posterior will follow if we show that $Q^{-\frac{1}{2}}A$ 
is continuous on a space $Y$ with appropriate norm $\|\cdot\|_Y$ and $y \in Y$ almost surely. 
Note that $G(u)$ lies in Cameron-Martin space of $\nu$ almost surely with respect to $\mu$. 
Hence, $y|u$ is absolutely continuous with respect to $\nu$ for almost all $u$ with respect to $\mu$. 
Thus, it is sufficient to show that $\eta \in Y$ almost surely.

We choose $Y = \zeta (G^T)^{-1}\mathcal{C}^{-\frac{1}{2}}\mathcal{H}_1$, with the norm $\|y\|_Y = \|\mathcal{C}^{\frac{1}{2}}G^T \zeta^{-1}y\|_1$. The random variable $\eta$ belongs to $Y$ almost surely, if $\|\eta\|_Y < \infty$ almost surely. This holds if 
$$\mathbb{E}\|\eta\|_Y^2 = \int \|\mathcal{C}^{\frac{1}{2}}G^T \zeta^{-1}y\|_1^2 d\nu < \infty,$$ 
which, in turn is true whenever the covariance operator of pushforward of $\nu$ under
$\mathcal{C}^{\frac{1}{2}}G^T \zeta^{-1}$, given by 
$\left({C}^{\frac{1}{2}}G^T\zeta^{-\frac{1}{2}}\right)\zeta^{-\frac{1}{2}}G\mathcal{C}^{\frac{1}{2}}$ is trace class.

Noting that $AB$ is trace class if and only if $BA$ is trace class, the above is equivalent to showing that 
$$\zeta^{-\frac{1}{2}}G\mathcal{C}^{\frac{1}{2}}\left({C}^{\frac{1}{2}}G^T\zeta^{-\frac{1}{2}}\right) 
= \zeta^{-\frac{1}{2}}G\mathcal{C}G^T\zeta^{-\frac{1}{2}}$$ 
is trace class. We have seen that this follows by the fact that $G(u)$ lies in the Cameron-Martin space of the noise almost surely with respect to the prior measure. Hence, $\eta \in Y$ almost surely.

Next, we show that $Q^{-\frac{1}{2}}A$ is continuous on $Y$.
For, each $y \in Y$, we have some $u \in \mathcal{H}_1$ such that $y = \zeta (G^T)^{-1}\mathcal{C}^{-\frac{1}{2}}u$ 
and $\|y\|_Y = \|u\|_1$. Thus,  $Q^{-\frac{1}{2}}A$ is continuous on $Y$ if 
$Q^{-\frac{1}{2}}A\zeta (G^T)^{-1}\mathcal{C}^{-\frac{1}{2}}$ is continuous on $\mathcal{H}_1$. Setting 
$\mathcal{C}^{\frac{1}{2}}G^T\zeta^{-\frac{1}{2}} = B$ and $G^T\zeta G = K $, we have
\begin{eqnarray*}
A\zeta (G^T)^{-1}\mathcal{C}^{-\frac{1}{2}} 
&=& \mathcal{C}G^T\zeta^{-\frac{1}{2}}\left(I + \zeta^{-\frac{1}{2}}G\mathcal{C}G^T\zeta^{-\frac{1}{2}}\right)^{-1}\zeta^{\frac{1}{2}} (G^T)^{-1}\mathcal{C}^{-\frac{1}{2}}\\
&=& \mathcal{C}^{\frac{1}{2}}B\left(I + B^TB\right)^{-1}B^{-1}
\end{eqnarray*}
and 
$$Q^{-1} = \mathcal{C}^{-1} + G^T\zeta G = \mathcal{C}^{-1} + K.$$ 
Next, we note that $B$ is compact (since $B^TB$ is trace class)
and the eigenbasis $\{g_i\}$ of $B^TB$ is also the eigenbasis of $T \equiv B\left(I + B^TB\right)^{-1}B^{-1}$. 
Assuming that $\{b_i^2\}$ are the eigenvalues of $B^TB$, the eigenvalues of $T$ are $\left(1 + b_i^2\right)^{-1}$, 
hence $T$ is continuous. Finally, we estimate 
$\|Q^{-\frac{1}{2}}A\zeta (G^T)^{-1}\mathcal{C}^{-\frac{1}{2}}u\|^2 $
Noting that $K$ is continuous, we have
\begin{eqnarray*}
\langle Q^{-1}A\zeta (G^T)^{-1}\mathcal{C}^{-\frac{1}{2}}u, A\zeta (G^T)^{-1}\mathcal{C}^{-\frac{1}{2}}u \rangle 
&=& \langle Tu,Tu \rangle + \langle K\mathcal{C}^{\frac{1}{2}}Tu,\mathcal{C}^{\frac{1}{2}}Tu \rangle\\
&\leq& a\|u\|^2
\end{eqnarray*}
for some $a>0$. Hence, $\|Q^{-\frac{1}{2}}Ay\|_1 \leq a\|y\|_Y^2$, proving the statement of theorem.

\textbf{The only if part}

We need to show that $\zeta^{-\frac{1}{2}}G\mathcal{C}G^T\zeta^{-\frac{1}{2}}$ is trace class, or equivalently, $\zeta^{-\frac{1}{2}}G\mathcal{C}^{\frac{1}{2}} = B^T$ is Hilbert-Schmidt. Well-posedness of the posterior implies that $Ay$ lies in the Cameron-Martin space of the posterior almost surely with respect to the distribution of $y$. This implies that $Q^{-\frac{1}{2}}A\zeta A^TQ^{-\frac{1}{2}}$ and $Q^{-\frac{1}{2}}AG\mathcal{C}G^TA^TQ^{-\frac{1}{2}}$ are trace class. In particular, this implies that $Q^{-\frac{1}{2}}AG\mathcal{C}^{\frac{1}{2}}$ and $fillin$ are Hilbert-Schmidt. Further, note that $Q^{-1}$ is bounded below since $\mathcal{C}^{-1}$ is bounded below and $\mathcal{C}^{-1}$ and $G^T\zeta G$ are positive operators. Thus, $Q^{-\frac{1}{2}}$ is bounded below making $AG\mathcal{C}^{\frac{1}{2}}$ Hilbert-Schmidt. Next, recalling that $Q^{-1} = \mathcal{C}^{-1} + G^T\zeta G$, we have $$\sum \langle Q^{-\frac{1}{2}}AG\mathcal{C}^{\frac{1}{2}}e_i,Q^{-\frac{1}{2}}AG\mathcal{C}^{\frac{1}{2}}e_i\rangle = \sum \langle (\mathcal{C}^{-1} + G^T\zeta G)AG\mathcal{C}^{\frac{1}{2}}e_i,AG\mathcal{C}^{\frac{1}{2}}e_i\rangle < \infty.$$
Since $\zeta^{\frac{1}{2}}GAG\mathcal{C}^{\frac{1}{2}}$ is Hilbert-Schmidt ($AG\mathcal{C}^{\frac{1}{2}}$ is Hilbert-Schmidt and $\zeta^{\frac{1}{2}}G$ is continuous), we have after explicitly writing out $A$, $$\sum \langle \mathcal{C}^{-1}AG\mathcal{C}^{\frac{1}{2}}e_i,AG\mathcal{C}^{\frac{1}{2}}e_i\rangle = \sum  \|\mathcal{C}^{\frac{1}{2}}G^T \zeta^{-\frac{1}{2}}\left(I + \zeta^{-\frac{1}{2}}G\mathcal{C}G^T \zeta^{-\frac{1}{2}}\right)^{-1}\zeta^{-\frac{1}{2}}G\mathcal{C}^{\frac{1}{2}}e_i\|^2 < \infty.$$
Hence, $B(I + B^TB)^{-1}B^T$ is Hilbert-Schmidt. Let  $B(I + B^TB)^{-1}B^T$ have eigenbasis $\{f_i\}$. It is easy to see that $B^TB$ is Hilbert-Schmidt and has eigenbasis $\{f_i\}$ as well. 

Now, we use the fact that $Q^{-\frac{1}{2}}A\zeta^{\frac{1}{2}}$ is Hilbert-Schmidt to show that $\zeta^{-\frac{1}{2}}G\mathcal{C}^{\frac{1}{2}}$ is Hilbert-Schmidt. By arguments used before, it follows that $A\zeta^{\frac{1}{2}}$ is also Hilbert-Schmidt. After opening up $Q^{-1}$, we have
$$\sum \langle Q^{-\frac{1}{2}}A\zeta^{\frac{1}{2}}e_i, Q^{-\frac{1}{2}}A\zeta^{\frac{1}{2}} e_i \rangle = \sum \langle (\mathcal{C}^{-1} + G^T\zeta G )A\zeta^{\frac{1}{2}}e_i, A\zeta^{\frac{1}{2}} e_i \rangle < \infty.$$
As before, noting that $\zeta^{\frac{1}{2}}GA\zeta^{\frac{1}{2}}$ is Hilbert-Schmidt and opening up $A$, we have $$\langle \mathcal{C}^{-1}A\zeta^{\frac{1}{2}}e_i, A\zeta^{\frac{1}{2}} e_i \rangle = \|\mathcal{C}^{\frac{1}{2}}G^T\zeta^{-\frac{1}{2}}\left(I + \zeta^{-\frac{1}{2}}G\mathcal{C}G^T \zeta^{-\frac{1}{2}}\right)^{-1} e_i\|^2 < \infty.$$ 
Hence, $B(I + B^TB)^{-1}$ is Hilbert-Schmidt. Assuming eigenvalues of $B^TB$ are $\{s_i^2\}$, we have $$\sum \|B(I + B^TB)^{-1}f_i\|^2 = \sum \left(\frac{s_i}{1+ s_i^2}\right)^2 < \infty.$$
$\sum \left(\frac{s_i}{1+ s_i^2}\right)^2 < \infty$ along with $s_i^2 \to 0$ (since $B^TB$ is compact) imply that $\sum s_i^2 < \infty$ making $B^TB$ trace class and hence $B^T$ Hilbert-Schmidt.
\end{proof}

Now, to show the result for Banach spaces, we shall embed the Banach spaces into appropriate Hilbert spaces and show that the well-posedness result on the Hilbert spaces imply the well-posedness result on the Banach spaces.
We rewrite model \ref{eqn:begin} in Banach space.
\begin{equation} \label{eqn:banach}
  y = G(u) + \frac{1}{\sqrt{n}}\eta \,,
\end{equation}
where $\eta$ is a Gaussian noise on $\calH_2$: $\eta \sim
\normal(0,\zeta)$. $G : W_1 \to W_2$ where $W_i$ are Banach spaces. $\mu$ is the prior on $W$. If there exists a Hilbert space $\mathcal{H}_1$ such that $W_1 \subset \mathcal{H}_1$, then we can push forward the prior $\mu$ via the inclusion map to get a measure $N(0, \mathcal{C}')$ on $\mathcal{H}_1$. Similarly, if there exists a Hilbert space $\mathcal{H}_2$ such that $W_2 \subset \mathcal{H}_2$ then we can rewrite model \ref{eqn:banach} as model \ref{eqn:begin} with minor changes. $G$ is a linear map on $\mathcal{H}_1$ which is defined almost everywhere with respect to the measure $N(0, \mathcal{C}')$. The noise stays the same as before. If we now write the posterior for this model prior, we get the same expression $N(Ay, Q)$ for the posterior and hence the well-posedness of model \ref{eqn:begin} implies the well-posedness of \ref{eqn:banach}. The only thing left to show is that given any Banach space $W$, we can find a Hilbert space $\mathcal{H}$ such that $W \subset \mathcal{H}$.

\begin{lemma}
Given a Banach space $W$, we can construct a separable Hilbert space $\mathcal{H}$ such that $W \subset \mathcal{H}$.
\end{lemma}
\begin{proof}
Since $W$ is separable, we can find a countable collection of linear functionals $f_i \in W^{*}$ with norm $1$ such that $\|x\|_W = sup_{i}|f_i(x)|$ for all $x \in W$. Define $$\langle x , y \rangle_{\mathcal{H}} = \sum_i \frac{f_i(x)f_i(y)}{i^2}.$$
Also, the $\|\|_{\mathcal{H}}$ norm is smaller than the $\|\|_{W}$ norm since $$\sum_i \frac{f_i(x)f_i(x)}{i^2} \leq sup_{i}|f_i(x)|^2 (\sum \frac{1}{n^2}) \leq K\|x\|_{W}.$$
Completing $W$ under the $\|\|_{\mathcal{H}}$ norm gives us the desired Hilbert space $\mathcal{H}$.
\end{proof}
This achieves the first part in Section \ref{contribution}. As an application, we shall discuss the well-posedness 
of posteriors for a certain class  of examples which strictly contains the examples discussed in \cite{agapiou}.

\begin{example}\label{ex:mild}
Consider the model \eqref{eqn:begin}, with the prior $\mu \equiv N(0,\mathcal{C})$. 
Let the operator $G = (\mathcal{C}^{-l} + K_1)^{-1}$ and noise 
$\eta \sim (\mathcal{C}^{-\frac{\beta}{2}} + K_2)^{-2}$ for some $K_1,K_2$ continuous, 
self adjoint, positive operators such that $\mathcal{C}^{s_0}$ is trace class for some 
$0 < s_0 < 1$ and $2l - \beta + 1 > s_0$. Assume further that 
$\mathcal{C}^{l - \frac{\beta}{2}}K_1$ is continuous.
\end{example}

\begin{claim}
The posterior for the above example is well-posed.
\end{claim}
\begin{proof}
We need to show that $\zeta^{-\frac{1}{2}}G(u) \in \mathcal{H}_2$ for almost all $u$ 
with respect to the prior measure. Following the previous proof, we can show that the 
above holds if $\zeta^{-\frac{1}{2}}G\mathcal{C}^s$ is continuous for some $s < \frac{1-s_0}{2}$. 
Hence, it is sufficient to show that $\zeta^{-\frac{1}{2}}G\mathcal{C}^{\frac{\beta}{2} - l}$ 
is continuous. Noting that $(I + K_2\mathcal{C}^{\frac{\beta}{2}})$ and 
$\left(I + \mathcal{C}^{l - \frac{\beta}{2}}K_1\mathcal{C}^{\frac{\beta}{2}}\right)^{-1}$ 
are continuous (by first part of theorem \ref{thm:functional} by putting 
$G_1 = \mathcal{C}^{l - \beta}$ and $K = \mathcal{C}^{\frac{\beta}{2}}K_1\mathcal{C}^{\frac{\beta}{2}}$), 
we have that 
$$\zeta^{-\frac{1}{2}}G\mathcal{C}^{\frac{\beta}{2} - l} 
= (I + K_2\mathcal{C}^{\frac{\beta}{2}})\left(I + \mathcal{C}^{l - \frac{\beta}{2}}K_1\mathcal{C}^{\frac{\beta}{2}}\right)^{-1}$$ 
is continuous. Hence, we have proved the claim.
\end{proof}

We note here that in the context of Example 8.3 in \cite{agapiou}, we allow for a larger class of perturbations
(class of $K_1$ and $K_2$ we can choose) both in noise measure and operator to be inverted.

\subsection{Contraction lemmas}
\label{contraction}

In this section, we shall focus on providing tight conditions which enable 
us in computing (almost) optimal contraction rates.
Each lemma in this section demands different
conditions, which are closely related to one another. Depending on the
kind of prior, one may find it easier or difficult to verify some of
these conditions and this may dictate which form of the lemma to
use. We shall present the proof only for the first contraction lemma,
and outline the proofs for the other two.

It has been noted by Knapik etal. \cite{Zanten} that it is possible to improve the contraction 
rates by choosing different priors $\mu_n$ for different noise
levels $n$. We shall also adopt the same method in our quest for better contraction rates in our setting.
All the priors $\mu_n$ however, shall be defined
using the same basis $\{\phi_i\}$. As described in Section
\ref{subsec-smallnoise}, contraction rate for
a posterior measure is a way of quantifying concentration of the
posterior measure around the true value.  However, such concentration
phenomenon is not exhibited by the posterior if the prior distribution does not put
enough probability mass around the true value. This can be avoided
with the following assumption:
\begin{assumption}\label{assum1}
  Assume that there exist a sequence of real numbers $\{\eps_n\}_{n\ge
    1}$ such that $\eps_n \to 0$ and $n\eps_n^2 \to\infty$, such that
  \begin{equation}\label{eqn:fwd-ball}
    \mu_n \{ u : \left\|G(u) - G(u_0)\right\|_{\zeta} \leq \epsilon_n
    \} \geq e^{-Cn\epsilon_n^2}\ \ \ \ \ \ \textrm{ for some } C>0,
  \end{equation}
where $\|\cdot\|_{\zeta}$ is the Cameron-Martin norm of the noise as defined
in \eqref{eqn:phiyu}, and $u_0$ is the true value.
\end{assumption}

Additionally, we also need to ensure that, with high probability (under the prior measure), elements in $\calH_1$ are well approximated
by the finite dimensional projections. This is made precise in the
following.
\begin{assumption}\label{assum2}
  Assume that $G$ admits a singular value decomposition, and that $G^TG$ has the eigenpair
  $\{e_k,\rho_k^2\}$. Further, assume that there exists a sequence of
  real constants $\xi_n > 0$, sequences of positive integers $k_n$,
  $r_n$ and a basis $\{\varphi_k\}_{k\ge 1}$ of $\calH_1$ satisfying
  \begin{itemize}
  \item $r_n\to\infty$ as $n\to\infty$;
  \item $k_n < Rn\eps_n^2$ for some $R>0$, where $\eps_n$ are as
    defined in the previous assumption;
  \item $\max\{\xi_n, k_n^{-1}\} \rightarrow 0$ as $n\to\infty$, and
  \item writing $P^{\varphi}_{k_n}$ and $P^{e}_{r_n}$ as the
    projections onto the subspace spanned by
    $\{\varphi_1,\ldots,\varphi_{k_n}\}$ and $\{e_1,\ldots, e_{r_n}\}$
    respectively, we need
    \begin{equation}\label{eqn:pr-proj}
      \mu_n \{u : \left\|P^{\varphi}_{k_n}P^e_{r_n}(u) - u\right\|_1 >
      C_2 \xi_n\} \leq e^{-(C + 4)n\epsilon_n^2}
    \end{equation}
    for some $C_2>0$, and the same $C$ as in Assumption \ref{assum1}.
  \end{itemize}
\end{assumption}

\begin{remark}
  Note that a sufficient condition to check the
  inequality in \eqref{eqn:pr-proj} is
  \begin{equation}\label{eqn:pr-proj1}
    \mu_n \left\{u : \left\|P^{\varphi}_{k_n}(u) - u\right\|_1 >
      \frac{C_2}{2} \xi_n\right\} \leq \frac{1}{2}e^{-(C +
      4)n\epsilon_n^2}
  \end{equation}
  and
  \begin{equation}\label{eqn:pr-proj2}
    \mu_n \left\{u : \left\|P^{\varphi}_{k_n} \left[ P^e_{r_n}(u) -
          u\right] \right\|_1 > \frac{C_2}{2} \xi_n\right\} \leq
    \frac{1}{2}e^{-(C + 4)n\epsilon_n^2}.
  \end{equation}
\end{remark}
Next is a technical assumption which underlines relationship
between the eigenbasis of $G^TG$ and the basis $\{\varphi_k\}$.
\begin{assumption}\label{assum3}
  Let $S \equiv (G^{-1})^T$, then define
  \begin{equation}\label{eqn:def-gr}
    g_{k,r} := \max_{\begin{subarray}{c} \|h\|_1 \leq 1 \\ h \in
        \text{span} \{\varphi_1,\ldots,\varphi_k\} \end{subarray}}
    \|\zeta^{1/2}SP_r^eh\|^2_2.
  \end{equation}
  We assume that $\sqrt{g_{k_n,r_n}} \leq C_1
  \frac{\xi_n}{\epsilon_n}$.  Note that $g_{k,r}$ is finite whenever
  $k$ and $r$ are finite.
\end{assumption}

In short, the choice of a priors $\mu_n$ and the operator $G$ will
determine the sequence $\{\eps_n\}$ in Assumption
\ref{assum1}.  Thereafter, $\{\eps_n\}$ along with $G$ and $\mu_n$
will dictate the existence of $\{k_n\}$, $\{r_n\}$ and $\{\xi_n\}$ appearing in
Assumptions \ref{assum2} and \ref{assum3}. Finally, equipped with the
above sequences, we shall see in the following lemma that $\xi_n$
is the rate of contraction of the posterior measure
$\mu_n^y$ defined in equation~\eqref{eqn:muyb}.

\begin{lemma}[Contraction Lemma]\label{lem:contraction}
  Consider the model given by equation~\eqref{eqn:begin}, together
  with Assumptions~\ref{assum1}-\ref{assum3} stated above. Also, let
  $u_0$ be such that $\left\|P^{\varphi}_{k_n}P^e_{r_n}(u_0) -
    u_0\right\|_1 = O(\xi_n)$. Then, for some $M > 0$,
  \begin{equation}\label{eqn:contraction}
    \mu_n^y ( u : \left\|u - u_0\right\|_1 > M \xi_n) \rightarrow 0
    \,\,\,\,\textrm{ in probability  }   \bbqon,
  \end{equation}
  where, recall that, $\xi_n$ is called the rate of contraction of the
  posterior measure $\mu^y_n$ around the true value $u_0$.
\end{lemma}

\begin{remark}
  Clearly, Assumption \ref{assum3} does not place any a priori
  restrictions on the basis of the prior as in previous works (see
  \cite{ray}).  We also note here that, at this stage we do not know
  if Assumption \ref{assum2} can be verified for a given problem when
  $r_n$ is finite.  However, in our next lemma we shall weaken
  Assumption \ref{assum2} by setting $r_n=\infty$.
\end{remark}

We shall prove the above lemma in an indirect way which follows the
work of Ghosal et. al. \cite{Ghosal}, where the authors established a
close relation between contraction rates and existence of sequence of
{\it tests}, which are real valued functions $\{\psi_n\}_{n\ge 1}$
defined on $\widehat{\calH}_2$ and satisfying certain regularity
conditions, which in turn translate into contraction rates.  More
precisely, we shall use the following proposition.

\begin{proposition}\label{prop:test}
  Let there exist {\bf tests} $\psi_n: \widehat{\calH}_2 \to\real$ for
  $n\ge 1$, satisfying
  \begin{eqnarray}
    \label{eqn:test-1}
    \sup_{ \{u\in\calH_1 : \left\|P_{k_n,r_n}(u) - u\right\|_1 \leq
      C_2 \xi_n,\left\|u - u_0 \right\|_1 \geq M \xi_n \} } \bbqun(1 -
    \psi_n)  &\leq & e^{-(C + 4)n\epsilon_n^2}\\
    \label{eqn:test-2}
    \bbqon(\psi_n) & \rightarrow & 0,
  \end{eqnarray}
where all the constants appearing above are as defined in Assumptions 
\ref{assum1}-\ref{assum3}.
Then $\xi_n$ is the contraction rate of the posterior measure
  $\mu^y_n$ around the true value $u_0$. Here $P_{k_n,r_n}$ is the
composition of $P^{\varphi}_{k_n}$ and $P^{e}_{r_n}$.
\end{proposition}

In view of the above proposition, which will be proved later, the
proof of Lemma \ref{lem:contraction} proceeds as follows: under
Assumptions \ref{assum1}-\ref{assum3}, we shall show that the same tests
$\psi_n$ as were used in the work of Ray \cite{ray} satisfy conditions
\eqref{eqn:test-1} and \eqref{eqn:test-2} above.

\begin{proof}[Proof of Lemma \ref{lem:contraction}.]
  Recall the model $y = G(u) + \frac{1}{\sqrt n} \eta$.  We will
  define
  $$ \tilde{\varphi}_{k,r} = \sum_{i=1}^{r}\frac{\left\langle
      \varphi_k, e_i\right\rangle}{\rho^2_i} Ge_i \overset{\Delta}{=}
  SP^e_r\varphi_k. $$\\
  Then, setting $\tilde{y}_{k,r}$ as the projection of $y$ onto
  $\tilde{\varphi}_{k,r}$, we can write
  $$ \tilde{y}_{k,r} = \left\langle
    \tilde{\varphi}_{k,r},G(u)\right\rangle +
  \frac1{\sqrt{n}}\tilde{\eta}_{k,r} = \left\langle u,
    P^e_r\varphi_k\right\rangle +
  \frac1{\sqrt{n}}\tilde{\eta}_{k,r},$$
  where $\tilde{\eta}_{k,r} = \langle
  \eta, \tilde{\varphi}_{k,r} \rangle$. Thereafter, we define
\begin{equation} \label{eqn:est-u}
\hat{u}_n = \sum_{k=1}^{k_n} \tilde{y}_{k,r_n} \varphi_k
\end{equation}
  Clearly, $\hat{u}_n = P^{\varphi}_{k_n}P^e_{r_n}u + \sum_{k =
    1}^{k_n}\frac1{\sqrt{n}}\tilde{\eta}_{k,r_n}$. Then we define the
  tests as
  \begin{equation}\label{eqn:def-test}
    \psi_n(y) = 1_{\{ \left\| \hat{u}_n - u_0 \right\|_1 \geq M_0
      \xi_n \}},
  \end{equation}
  where $M_0$ is some non-negative constant to be identified later.

  In order to prove that the proposed tests satisfy conditions
  \eqref{eqn:test-1} and \eqref{eqn:test-2} we shall begin by
  estimating the distribution of $\hat{u}_n$ around its mean. To this
  end, using the separability of $\calH_1$ and the Hahn-Banach
  theorem, there exists a dense subset $B_1$ of a unit ball in
  $\calH_1$ such that for every element $v\in\calH_1$,
  $$\|v\|_1 = \sup_{w\in B_1} \left| \langle v,w\rangle_1\right|$$
  Interpreting the norm as the supremum over a dense set as above, and
  using the Borell-TIS inequality \cite[Theorem 2.1.1]{RFG} we get,
  \begin{equation}\label{eqn:Borell}
    \bbq_{u,n} \left[ \left\| \hat{u}_n - \bbq_{u,n} \left( \hat{u}_n
        \right) \right\|_1 - \bbq_{u,n} \left( \left\| \hat{u}_n -
          \bbq_{u,n} \left( \hat{u}_n\right) \right\|_1 \right) \geq x
    \right] \leq e^{-\frac{x^2}{2\sigma_0^2}}, \,\,\,\,\,\forall x > 0
  \end{equation}
  where $\sigma_0$ is given as
  \begin{equation}\label{eqn:sigma2}
    \sigma_0^2 = \sup_{h \in B_1}  \bbq_{u,n}\left( \frac{1}{\sqrt{n}}
      \sum_{k = 1}^{k_n} \tilde{\eta}_{k,r_n} \left \langle h, \varphi_k
      \right \rangle_1 \right)^2 . 
  \end{equation}
  \\
  In order to get appropriate estimates on $\hat{u}_n$, we need to now
  estimate $\sigma_0$ and $\bbq_{u,n}\left(\left\| \hat{u}_n -
      \bbq_{u,n} \left(\hat{u}_n\right)\right\|_1\right)$. Notice that
  by Jensen's inequality, we get
  \begin{align*}
    \left[\bbq_{u,n} \left(\left\| \hat{u}_n - \bbq_{u,n} \left(
            \hat{u}_n \right) \right\|_1 \right) \right]^2 &\le
    \frac1n \sum_{k = 1}^{k_n} \bbq_{u,n} \left( \tilde{\eta}_{k,r_n}^2
    \right) \\
    &= \frac1n \sum_{k = 1}^{k_n} \left\| \zeta^{1/2}
      \tilde{\varphi}_{k,r_n} \right\|_2^2 = \frac1n \sum_{k =
      1}^{k_n} \left\| \zeta^{1/2} S P^e_{r_n} \varphi_k \right\|_2^2
  \end{align*}
  Then, recalling the definition of $g_n$ from Assumption
  \ref{assum3}, we have
  \begin{align}\label{eqn:mean-deviation}
    \left[ \bbq_{u,n} \left( \left\| \hat{u}_n - \bbq_{u,n} \left(
            \hat{u}_n \right) \right\|_1 \right) \right]^2 \leq
    \frac{k_n}{n} \left[ \max_{\begin{subarray}{c} \|h\|_1 = 1 \\ h \in
        \text{span} \{\varphi_1, \ldots,
        \varphi_{k_n}\} \end{subarray}} \left\| \zeta^{1/2} S
      P_{r_n}^e h \right\|^2_2 \right] \leq \frac{k_n}{n}
    g_{k_n,r_n}
  \end{align}
  Thereafter, to estimate $\sigma_0^2$, we start with the following
  \begin{align*}
  \bbq_{u,n}\left( \frac{1}{\sqrt{n}}
      \sum_{k = 1}^{k_n} \tilde{\eta}_{k,r_n} \left \langle h, \varphi_k
      \right \rangle_1 \right)^2  
    &= \frac1n \bbq_{u,n}\left\langle
\eta,\sum_{k
        = 1}^{k_n}\left\langle h,\varphi_k\right\rangle
      \tilde{\varphi}_{k,r_n}\right\rangle_2^2 \\
    &= \frac1n \left\|\zeta^{1/2}\left(\sum_{k = 1}^{k_n}\left\langle
          h,\varphi_k\right\rangle
        \tilde{\varphi}_{k,r_n}\right)\right\|_2^2\\
    &= \frac1n
    \left\|\zeta^{1/2}SP^e_{r_n}P^{\varphi}_{k_n}(h)\right\|_2^2.
  \end{align*}
  Therefore, equation \eqref{eqn:sigma2} can be restated as
  \begin{equation}\label{eqn:sigma2-a}
    \sigma_0^2 = \sup_{h \in B_0} \frac1n \left\| \zeta^{1/2} S
      P^e_{r_n} P^{\varphi}_{k_n}(h) \right\|_2^2 =
    \frac{g_{k_n,r_n}}{n}
  \end{equation}
  \\
  Substituting the conclusions of equations~\eqref{eqn:mean-deviation}
  and \eqref{eqn:sigma2-a} in the inequality~\eqref{eqn:Borell} and
  choosing $x^2 = 2L\epsilon_n^2g_{k_n,r_n}$, for some $L >0$. we have for large enough $M$ (depending on $L$),
  \begin{equation}
    \bbq_{u,n} \left( \left\| \hat{u}_n - \bbq_{u,n} \left( \hat{u}_n
        \right) \right\|_1 \geq M \epsilon_n \sqrt{g_{k_n,r_n}}
    \right) \leq \exp(-L n \epsilon_n^2),
  \end{equation}
  which is analogous to equation 4.2 in \cite{ray} with
  $\sqrt{g_{k_n,r_n}}$ replacing $1/\delta_{k_n}$. Notice that the
  above estimate does not depend on $u$, which is not coincidental but
  a simple consequence of Gaussian computation.

  Using the above estimate and following the Proof of Theorem 2.1 in
  \cite{ray}, almost line by line, we can prove that the proposed
  tests satisfy conditions \eqref{eqn:test-1} and
  \eqref{eqn:test-2}, which together with Proposition \ref{prop:test}
  proves the contraction lemma.
\end{proof}

Next, we shall prove Proposition \ref{prop:test}. It is easy to see
that both the numerator as well as the denominator in equation
\eqref{eqn:muyb} for the posterior measure of a set converge to
$0$ as $n \rightarrow \infty$. Therefore, in order to extract the
exact rate, we shall analyse the numerator and denominator separately.

\begin{proof}[Proof of Proposition \ref{prop:test}.]
  As before, writing $\bbq_{u_0,n}$ as expectation with respect to the
  noise measure shifted by $G(u_0)$, notice that for any Borel subset
  $B$ of $\calH_1$, we have
  \[
  \bbq_{u_0,n} \left(\mu^y_n(B)\psi_n\right) \leq \bbq_{u_0,n} \psi_n
  \rightarrow 0 \ \ \ \ \ \textrm{ by condition~\eqref{eqn:test-2}}.
  \]
  Multiplying the numerator and denominator of
  equation~\eqref{eqn:muyb} with a constant \\$\exp \left( \frac{n \|
      G(u_0) \|_{\zeta}^2}{2} - n \left\langle G(u_0), y
    \right\rangle_{\zeta} \right)$ that depends on $u_0$ but is
  independent of $u$, allows us to write the posterior measure as
  $$\mu^y_n(B) = \frac{1}{Z_{u_0}} \int_B\exp\left(-\frac{n\|
      G(u)\|_{\zeta}^2}{2} + \frac{n\| G(u_0)\|_{\zeta}^2}{2} +
    n\left\langle G(u - u_0),y\right\rangle_{\zeta}\right)d\mu_n(u)
  \,.$$
  Note that the above algebraic manipulation can be done for a.e.-$y$ (under the noise measure).  
  Next, denoting $W_n^y(B)$ as the
  numerator of $\mu^y_n(B)$ above, we now estimate
  $\bbq_{u_0,n}\left[(1 - \psi_n)W^y_n(B)\right]$.  Applying Fubini
  theorem we observe that
  \begin{align*}
    & \bbq_{u_0,n} \left( \int_B (1 - \psi_n) \exp \left( -\frac{n \|
          G(u)\|_{\zeta}^2}{2} + \frac{n \| G(u_0)\|_{\zeta}^2}{2} + n
        \left\langle G(u - u_0), y \right\rangle_{\zeta} \right) d
      \mu_n  \right) \\ &= \int_B  \bbq_{u_0,n} \left( (1 - \psi_n)
      \exp \left( -\frac{n \| G(u) \|_{\zeta}^2}{2} + \frac{n \|
          G(u_0) \|_{\zeta}^2}{2} + n \left\langle G(u - u_0), y
        \right\rangle_{\zeta} \right) \right) d\mu_n \\ &= \int_B
    \bbq_{u,n} \left( 1 - \psi_n \right) d\mu_n \,,
  \end{align*}
  where we have used the Gaussian change of measure formula relating
  $\bbq_{u_0,n}$ and $\bbq_{u,n}$ to get the last equality. Then setting
$B = \{u\in\calH_1 : \left\|P^{\varphi}_{k_n}(u) -
    u\right\|_1 \leq C_2 \xi_n,\left\|u - u_0 \right\|_1 \geq M \xi_n
  \}$, and using inequality~\eqref{eqn:test-1}, we get
  $$\bbq_{u_0,n} \left[ (1 - \psi_n) W^y_n(B) \right] \leq \exp
  \left(-(C + 4) n \epsilon_n^2 \right).$$
  Hence,
  \begin{align*}
    & \bbq_{u_0,n} \left( (1 - \psi_n) W^y_n (\{u : \left\| u - u_0
      \right\|_1 \geq M \xi_n \}) \right) \\ & \leq \exp \left(-(C + 4) n
      \epsilon_n^2 \right) + \mu_n \{ \left\| P_{k_n}(u) - u
    \right\|_1 > C_2 \xi_n \}
  \end{align*}
  Next, we analyse the denominator of equation~\eqref{eqn:muyb}. Let
  $D = \{ u : \left\|G(u)- G(u_0)\right\|_{\zeta} \leq \epsilon_n \},$
  and $\nu$ be a probability measure on D. We shall try to estimate
  the integrand (in the denominator) on this set $D$. To start with,
  we see that
  \[ -\frac{\| G(u)\|_{\zeta}^2}{2} + \frac{\| G(u_0)\|_{\zeta}^2}{2}
  + \left\langle G(u - u_0),y\right\rangle_{\zeta} = -\frac{\| G(u -
    u_0)\|_{\zeta}^2}{2} + \left\langle G(u - u_0),y -
    G(u_0)\right\rangle_{\zeta} \]
  Then,
  \begin{align*}
    & \bbq_{u_0,n}\left(\int_D (-\frac{\| G(u - u_0)\|_{\zeta}^2}{2} +
      \left\langle G(u - u_0),y - G(u_0)\right\rangle_{\zeta})d\nu
      \leq -(1 + C)\epsilon_n^2 \right)\\
    &\leq  \bbq_{u_0,n}\left(\int_D (\left\langle G(u - u_0),y -
        G(u_0)\right\rangle_{\zeta})d\nu \leq - (1/2 + C)\epsilon_n^2
    \right)\\
    &= \bbq_{u_0,n}\left(\int_D (\left\langle G(u -
        u_0),Z\right\rangle_{\zeta})d\nu \leq -(1/2 + C)\epsilon_n^2
    \right)\\
  \end{align*}
  Now, using Chebyshev's and Jensen's inequalities successively, we
  get
  \begin{align*}\label{eqn:den}
    & \bbq_{u_0,n}\left(\int_D (\left\langle G(u -
        u_0),Z\right\rangle)d\nu \leq -(1/2 + C)\epsilon_n^2
    \right)\\
    &\leq \frac{\int_D \bbq_{u_0,n}\left\langle Z, G(u -
        u_0)\right\rangle^2_{\zeta}d\nu}{(-(1/2 + C)\epsilon_n^2 )^2}
    \leq \frac{\epsilon_n^2}{n(-(1/2 + C)\epsilon_n^2 )^2}\\ &\leq
    K\,\left(n\epsilon_n^2\right)^{-1}
  \end{align*}
  for some $K > 0.$ Next, let us define
  $$S_n = \left\{ y : \int_D
    \left(-\frac{\| G(u - u_0)\|_{\zeta}^2}{2} + \left\langle G(u -
        u_0), y \right\rangle_{\zeta} \right) d\nu \geq -(1 +
    C)\epsilon_n^2 \right\} \,,$$
  and note that $\bbq_{u_0,n}(S_n) \to 1$.

  It is easy to see that the numerator and the denominator of
  $\mu^y_n(\{u :\left\|u - u_0 \right\|_1 \geq M \xi_n \})(1 -
  \psi_n)$ converge in probability (w.r.t. $\bbq_{u_0,n}$) to $0$ at
  rates $\exp\left(-(C + 4)n\epsilon_n^2\right)$ and $\exp\left(-(C +
    2)n\epsilon_n^2\right)$, respectively.  This implies that
  $\mu^y_n(\{u :\left\|u - u_0 \right\|_1 \geq M \xi_n \})$ converges
  in probability (w.r.t. $\bbq_{u_0,n}$) to $0$.
\end{proof}

\subsection{Contraction lemma for a restricted model}

In Lemma~\ref{lem:contraction}, we needed to get estimates for prior
measure along $\{e_k\}$ only on $k_n$ dimensional subspaces along the
prior oriented axes.  In case it is difficult to obtain estimates as in equation \eqref{eqn:pr-proj},
we need certain conditions on $\{\varphi_k\}$ with respect to $G$
to help us get rid of estimating any tail cylinder sets along axes
other than what the prior is oriented along. It is not difficult to check that this 
can be achieved by setting $r_n = \infty$. We now state
the new set of assumptions.

Assumption \ref{assum1} stays the same, and in lieu of
Assumption \ref{assum2}, we have the following assumption.

\begin{assumption}\label{assum2''}
  Assume that there exist constants $\xi_n$, sequences of positive
  integers $k_n$, and a basis $\{\varphi_n\}_{n\ge 1}$ of $\calH_1$
  satisfying
  \begin{itemize}
  \item $k_n < Rn\eps_n^2$ for some $R>0$ and $\eps_n$ as in the
    previous assumption;
  \item $\max\{\xi_n, k_n^{-1}\} \rightarrow 0$ as $n\to\infty$, and
  \item writing $P_{k_n}$ as the projection onto the subspace spanned
    by $\{\varphi_1,\ldots,\varphi_{k_n}\}$, we need
    \begin{equation}\label{eqn:pr-proj4}
      \mu_n \{u : \left\|P_{k_n}(u) - u\right\|_1 > C_2 \xi_n\} \leq
      e^{-(C + 4)n\epsilon_n^2}
    \end{equation}
  \end{itemize}
\end{assumption}
Further, instead of assumption \ref{assum3}, we have
\begin{assumption}\label{assum3''}
  Define $g_k$ by
  \begin{equation}\label{eqn:def-gr1}
    g_{k} := \max_{\begin{subarray}{c} \| h \|_1 \le 1 \\ h \in
        \text{span} \{ \varphi_1, \ldots, \varphi_k\} \end{subarray}}
     \!\!\!\!\!   \|\zeta^{1/2}Sh\|_2^2 \,\,\,
=    \!\!\!\!\! \max_{\begin{subarray}{c} \| h \|_1 = 1 \\ h \in
        \text{span} \{ \varphi_1, \ldots, \varphi_k\} \end{subarray}}
       \!\!\!\!\! \|\zeta^{1/2}Sh\|_2^2
  \end{equation}
  Finiteness of $g_k$ can be guaranteed by imposing the condition that
  $S\varphi_j \in \calH_2$ for all j.  We assume that
  $\sqrt{g_{k_n}} \leq C_1 \frac{\xi_n}{\epsilon_n}$.
\end{assumption}
Note that as a benefit of setting $r_n = \infty$, we do not need to explicitly
know the singular value decomposition of $G$ to obtain contraction rates.
This brings us to the following result on contraction rates in the aforementioned setup.

\begin{lemma}\label{lemma r infinite}
  Consider the model given by \eqref{eqn:begin}, together with
  assumptions~\ref{assum1},\ref{assum2''},\ref{assum3''} stated
  above. Also, let $u_0$ be such that 
  \begin{equation}\label{eqn:u0-assum}
  \left\|P^{\varphi}_{k_n}(u_0) -  u_0\right\|_1 = O(\xi_n)
  \end{equation} 
  then,
  \begin{equation}\label{eqn:contraction3}
    \mu_n^y ( u : \left\|u - u_0\right\|_1 > M \xi_n) \rightarrow 0
    \,\,\,\,\textrm{ in probability} ,
  \end{equation}
\end{lemma}
We omit the proof of this version of the lemma as it proceeds along the lines
as in the proof of Lemma \ref{lem:contraction} using the same test functions.
This achieves the second part in subsection \ref{contribution}. We note that it is this lemma that we shall be
using in the examples below.


\subsection{A version of the Lemma \ref{lem:contraction}}

We shall now discuss a lemma which applies to certain semilinear cases. The proof is very similar to that of previous lemma. 

The model  here is
\begin{equation}
    y = G(u^{*}) + \frac{1}{\sqrt{n}} \eta.
\end{equation}
For simplicity, we take $\varphi_k = e_k$ and replace $u$ by $u^*$ in this lemma where $u^*$ is a bijective 
continuous transform of $u$ with a continuous inverse.
We shall follow the same proof and will take $e_i = \phi_i$. We shall adopt assumptions wherever necessary. Further, the equivalent of true solution shall be taken as $u_0^*$ instead of $u_0$.

 As in \eqref{eqn:est-u}, we construct an estimator for $u^{*}$ given by
\begin{equation*}
    \hat{u}_{n}(y) = P_{k_n}u^{*} + \frac{1}{\sqrt{n}}\sum_{k =
1}^{k_n} \tilde{\eta}_k.
\end{equation*}
$P_{n}$ is the projection operator on the basis $\{\varphi_k\} = \{e_k\}$.
We define the test function as
\begin{equation*}
    \psi_n(y) = 1_{\{\|\hat{u}_n - u_0^{*}\|_1 \geq M\xi_n\}}.
\end{equation*}
Proceeding as we would  have in the proof of Lemma \ref{lem:contraction}, we get
\begin{equation*}
\mathbb{Q}_{u,n}(\|\hat{u}_n - \mathbb{Q}_{u,n}(\hat{u}_n)\|_1 \geq
M\epsilon_ng_{k_n})\leq \exp{(-Ln\epsilon_n^2)},
\end{equation*}
where $\mathbb{Q}_{u,n}$ is the noise shifted by $G^{*}(u) \equiv G(u^*)$. Since we
assume that $\{e_i\}$ and $\{\phi_i\}$ are the same, assumption \ref{assum3} is trivially satisfied. 
Using the above equation and proceeding as in \cite{ray}, the following properties are easily verified:
\begin{eqnarray*}
    \mathbb{Q}_{u_0,n}\psi_n &\to& 0\\
    \sup_{\{ u \in \mathcal{H}_1 : \|P_{k_n}(u^{*}) - u^{*}\| \leq
C_2\xi_n, \|u - u_0\|_1 \geq M\xi_n\}} \mathbb{Q}_{u,n}(1 -
\psi_n) &\leq& \exp(-(C + 4)n\epsilon_n^2)\\
\end{eqnarray*}
where we have assumed that $\|P_{k_n}(u_0^{*}) - u_0^{*}\| \leq \xi_n$.

Note that by continuity (and continuous inverse) of the transform $u^{*}$, 
we are able to keep the condition $\|u - u_0\|_1 \geq
M\xi_n$ instead of changing it to $\|u^{*} - u^{*}_0\|_1 \geq M\xi_n$. 

Proceeding further in the proof, we first look at the denominator of the expression for posterior measure. In the
proof, we constructed a set $S_n$ with $\mathbb{Q}_{u_0,n}(S_n) \to 1$ and
find a uniform (w.r.t. $y$) lower bound on the denominator. For this, we
restricted the integration (wrt to the prior $\mu_n$) to the set $\{u :
\|G(u) - G(u_0)\| \leq \epsilon_n\}$. In the present case, we replace the
above set by $\{u : \|G(u^{*}) - G(u_0^{*})\| \leq \epsilon_n\}$ and the
denominator has the same uniform lower bound on $S_n$ as in proposition
3.7. Thus, we need to modify assumption \ref{assum1} to the following.
\begin{assumption}\label{assum1'}
Assume that there exist a sequence of real numbers $\{\eps_n\}_{n\ge
    1}$ such that $\eps_n \to 0$ and $n\eps_n^2 \to\infty$, such that
  \begin{equation}\label{eqn:fwd-ball}
    \mu_n \{ u : \left\|G(u^{*}) - G(u^{*}_0)\right\|_{\zeta} \leq \epsilon_n
    \} \geq e^{-Cn\epsilon_n^2}\ \ \ \ \ \ \textrm{ for some } C>0.
  \end{equation}
\end{assumption}
We can follow the steps for estimating the numerator of the posterior as in the proof  of Lemma \ref{lem:contraction} to get 
the exact same estimates for the numerator if we replace assumption \ref{assum2} by the following assumption
\begin{assumption}\label{assum2'}
  Assume that $G^TG$ has the eigenpair $\{e_k,\rho_k^2\}$.  Further, assume that there exist constants $\xi_n$, sequences of positive integers $k_n$
  \begin{itemize}
  \item $k_n < Rn\eps_n^2$ for some $R>0$ and $\eps_n$ as in the
    previous assumption;
  \item $\max\{\xi_n, k_n^{-1}\} \rightarrow 0$ as $n\to\infty$, such that the following holds.
   \begin{equation}\label{eqn:pr-proj3}
      \mu_n \{u : \left\|P{k_n}(u^{*}) - u^{*}\right\|_1 > C_2 \xi_n\} \leq
      e^{-(C + 4)n\epsilon_n^2}
    \end{equation}
  \end{itemize}
\end{assumption}
Having estimated the numerator and denominator in the expression for posterior, the following lemma is readily proved.
\begin{lemma}\label{corollary1}
  Consider the model given by \eqref{eqn:begin}, together with
  Assumptions~\ref{assum1},\ref{assum2'} stated
  above. Also, let $u_0$ be such that $\left\|P_{k_n}(u^{*}_0) -
    u^{*}_0\right\|_1 = O(\xi_n)$ then,
  \begin{equation}\label{eqn:contraction2}
    \mu_n^y ( u : \left\|u - u_0\right\|_1 > M \xi_n) \rightarrow 0
    \,\,\,\,\textrm{ in probability} ,
  \end{equation}
\end{lemma}

We will use this lemma to get contraction rates for a semilinear example.

\section{Examples} \label{sec-examples}

The examples we shall be discussing in this section fit in the setup of
Lemma \ref{lemma r infinite}. Notice that the requirement $\epsilon_n
\sqrt{g_{k_n}} < C_1 \xi_n$ is one of the main
constraints in proving this lemma. 
This condition dictates the kind of leeway we can get with the basis on
which the prior is defined. Further, the conditions that we
need to verify may also demand certain further relations between various
bases. The first two examples involve inversion of a mildly ill-posed linear 
operator (\ref{M}) with Gaussian priors. The third example deals with inversion 
of a severely-ill posed problem with Gaussian priors. 
The fourth example involves the inversion of a semilinear operator with Gaussian priors. 
The final example is an inversion of a linear operator with compactly supported prior.

Recall that the recipe to obtain the contraction rates involves following steps: first, we shall
provide a method to find appropriate $\{\epsilon_n\}$
which satisfy Assumption \ref{assum1}; then, we find appropriate $k_n$ and
$\xi_n$ satisfying Assumption \ref{assum2''}; thereafter, we check if we
have appropriate $g_n$ which satisfy Assumption \ref{assum3''};
lastly, we check if the true solution $u_0$ satisfies the assumption \ref{eqn:u0-assum} stated in Lemma \ref{lemma r infinite}.

\subsection{Gaussian priors}

We shall begin by estimating a lower bound for $\mu_n \{ u : \left\|G(u) - G(u_0)\right\|_{\zeta} \leq \epsilon\}$ 
and hence getting appropriate $\{\epsilon_n\}$.
Estimating this quantity can be a difficult task whenever the operators involved ($G$, covariance of the 
prior and covariance of the noise) in the computation have different eigenbases. We may choose to 
consider a smaller set to get this estimate, if it is convenient.
The general procedure is as follows. First, we push $\mu_n$ via the
map $\zeta^{-1/2}G$ to get the Gaussian measure $\nu_n$. We need to
evaluate 
\begin{equation} \label{eqn:log-small-ball} 
\phi(\nu_n,G(u_0),\epsilon) \equiv -\log\left(\nu_n\{z\in\calH_2 : \|z - G(u_0)\| \leq \epsilon\}\right).\end{equation}
Using Cameron Martin theorem, followed by Jensen's inequality, we get
\begin{equation}\label{eqn:rate-lwr-bd}
\phi(\nu_n,G(u_0),\epsilon) \leq \phi(\nu_n,0,\epsilon/2) +
\mathcal{K}(\nu_n, G(u_0), \epsilon)
\end{equation}
where $$\mathcal{K}(\nu_n, G(u_0), \epsilon) = \underset{y \in \calH_2
\|y - G(u_0)\| \leq \epsilon/2}{\inf} \ \ \ \|y\|_{\nu_n}^2$$
 where $\|\cdot\|_{\nu_n}$ is the Cameron Martin norm with respect to the measure $\nu_n$.
Note that the first term on right hand side of \eqref{eqn:rate-lwr-bd} depends only on the eigenvalues of $\nu_n$, while the
second term depends on expansion of $G(u_0)$ in the eigenbasis of
$\nu_n$ as well as eigenvalues of $\nu_n$. We shall be dealing with these
calculations in the appendix.

We now outline the general steps to be followed to calculate contraction rates using Lemma \ref{lemma r infinite}
\begin{itemize}
\item[1)] We calculate $\phi(\nu_n,G(u_0),\epsilon)$ using \eqref{eqn:rate-lwr-bd}, and find $\epsilon_n$ such that 
$$\phi(\nu_n,G(u_0),\epsilon_n) \geq n\epsilon_n^2.$$ 
Such $\epsilon_n$ will satisfy Assumption \ref{assum1}. Equation \eqref{eqn:rate-lwr-bd} has two summands to be estimated:
\begin{itemize}
\item[a)] $\phi(\nu_n, 0 ,\epsilon)$ is estimated using Corollary \ref{mild}, for which we need estimates for eigenvalues of $\nu_n$.
\item[b)] The bias term $\mathcal{K}(\nu_n, G(u_0), \epsilon)$ is calculated using Theorems \ref{non-cent-small-ball-1}, \ref{non-cent-small-ball-2} 
or \ref{non-cent-small-ball-3} depending on the examples to be discussed.
\end{itemize}
\item[2)] We then calculate appropriate $k_n$ and $\xi_n$ which satisfy Assumption \ref{assum2''}.
\item[3)] We calculate $g_n$ and check if $\epsilon_n$ and $\xi_n$ satisfy the Assumption \ref{assum3''}.
\item[4)] We check that $\xi_n$ satisfies Assumption \ref{eqn:u0-assum}.
\end{itemize}

We note here that all the above quantities except $g_n$ depend on the scale parameter $R_n$ 
which we shall choose so as to improve the contraction rate $\xi_n$.

 Before we start with explicit examples, we state a form of
$\emph{min-max}$ theorem (\cite{Barry}, Theorem $XIII$) which we shall be using in the examples.

 \begin{theorem}[Min-max Theorem] \label{thm:minimax} 
 Assume $\mathcal{L}$ is a compact, self adjoint and positive definite
operator on some Hilbert space $\calH$. Let its eigenvalues be $\alpha_j$
and $\alpha_j \geq \alpha_{j + 1}$. Then the following equalities hold
 \begin{itemize}
 \item $\alpha_j = \underset{\{S : \textnormal{dim} (S) = j\}}{\max} \ \underset{\{x :
\|x\| = 1, x \in S\}}{\min} \langle\mathcal{L}x,x\rangle_{\calH}$
 \item $\alpha_j = \underset{\{S : \textnormal{codim} (S) = j - 1\}}{\min} \
\underset{\{x : \|x\| = 1, x \in
S\}}{\max}\langle\mathcal{L}x,x\rangle_{\calH}$
 \item $\frac{1}{\alpha_j} = \underset{\{S : \textnormal{codim} (S) = j - 1\}}{\max} \
\underset{\{x : \|x\| = 1, x \in
S\}}{\min}\langle\mathcal{L}^{-1}x,x\rangle_{\calH}$
 \item $\frac{1}{\alpha_j} = \underset{\{S : \textnormal{dim} (S) = j\}}{\min} \
\underset{\{x : \|x\| = 1, x \in
S\}}{\max}\langle\mathcal{L}^{-1}x,x\rangle_{\calH}$
 \end{itemize}
 \end{theorem}

\subsubsection{Example 1: Deconvolution with Meyer wavelet priors}\label{sec:meyer}

The details of deconvolution problem with Meyer wavelet prior can be found in\cite{ray} and the references therein. 
We use Gaussian priors instead of uniform priors as in \cite{ray}. The specifics of model we use here are as follows.
\begin{itemize}
\item The noise $\eta$ is white.
\item The model \eqref{M} is mildly illposed. We assume that the exponent of ill posedness is $\alpha$, and the eigenbasis of $G$ is the standard fourier basis (denoted by $\{e_i\}$).
\item The Meyer wavelet basis (denoted by $\phi_i$) satisfies the property $$\left\langle\varphi_j,e_k\right\rangle_1 = 0$$ if $k \notin [ j/3 - 1, 2j
].$
\end{itemize}
We further assume that the prior $\mu_n$ is Gaussian with covariance operator
$\calC_n$ and eigenpair $\{\varphi_j, \frac{\lambda_j}{R_n^2}\}$, where $\lambda_j = j^{-1 - 2\delta}$, and $R_n>0$, for $n\ge 1$, is a scale parameter.
\begin{theorem}
The model described above is well-posed. Further, the contraction rates $\xi_n$  for true solutions $u_0 \in \calH_1^{\gamma}(\calC_n)$ are given by the following expressions. 
$$\xi_n = 
                  \begin{cases}
                  n^{-\frac{\gamma}{1 + 2\alpha + 2\gamma}} & 2\gamma \leq   1+ 2\delta, R_n = n^{\frac{\gamma - \delta}{1 + 2\alpha + 2\gamma}}\\
                  n^{-\frac{2\delta + 1}{4(1 + \alpha + \delta)}} & 2\gamma > 1 + 2\delta, R_n = \frac{1}{4(1 + \alpha + \delta)}\\
                  \end{cases}$$
\end{theorem}
\begin{proof}
 since $\zeta = I$ and $G$ is continuous, the posterior is well-posed by Theorem \ref{thm:well-posedness}. It also follows that
\begin{equation}\label{eqn:gn-meyer}
g_n \leq (2n)^{2\alpha}.
\end{equation}

This makes $\nu_n$ a centered Gaussian measure on $\calH_2$ whose
covariance operator is $G\mathcal{C}_nG^T$
with eigenvalues $\{\alpha_{j,n}\}$.
For the sake of brevity, let us define $D_j^1$ and $D_j^2$ to be the span
of $\{e_1,e_2.....e_j\}$ and  $\{Ge_1,Ge_2.....Ge_j\}$ respectively. Next,
to obtain the eigenvalues $\{\alpha_{j,n}\}_{j\ge 1}$ of $G\mathcal{C}_n G^T$,
we apply the Min-max Theorem \ref{thm:minimax}
\begin{align*}
\alpha_{j,n} &= \sup_{C : C \subset \calH_2, \textnormal{dim} C = j} \ \inf_{x : x \in
C} \frac{\left\langle G\mathcal{C}_n G^T x,x\right\rangle_2}{\left\langle
x,x\right\rangle_2}\\
&\geq \inf_{x \in D_j^2} \frac{\left\langle \mathcal{C}_n G^T
x,G^Tx\right\rangle_1}{\left\langle
G^Tx,G^Tx\right\rangle_1}\frac{\left\langle GG^T
x,x\right\rangle_2}{\left\langle x,x\right\rangle_2}\\
&\geq \rho_j^2 \inf_{x \in D_j^2} \frac{\left\langle \mathcal{C}_n G^T
x,G^Tx\right\rangle_1}{\left\langle G^Tx,G^Tx\right\rangle_1}\\
&=  \rho_j^2  \inf_{x \in D_j^1, \vert\vert x\vert\vert = 1} {\left\langle
\mathcal{C}_n x,x\right\rangle_1}
\end{align*}
For $x\in D^1_j$, let $x = \sum x_k \varphi_k$. Then, we have $x_k = 0$
for $k > 2j$, since $\left\langle\varphi_j,e_k\right\rangle_1 = 0$ if $k
\notin [ j/3 - 1, 2j ]$. This implies
\begin{equation*}
\inf_{x \in D_j^1, \vert\vert x\vert\vert = 1} {\left\langle
\mathcal{C}_n x,x\right\rangle_1} = \sum_{k = 1}^{2j}
\frac{\lambda_k}{R_n^2} x_k^2 \geq \frac{\lambda_{2j}}{R_n^2}.
\end{equation*}
Hence, we have $\alpha_{j,n} \geq \rho_j^2 \frac{\lambda_{2j}}{R_n^2} \geq
C_1 \rho_j^2 \frac{\lambda_j}{R_n^2}$. In an exactly similar fashion, 
by using the other half of Min-max theorem, we can show that
$\alpha_{j,n} \leq C_2 \rho_j^2 \frac{\lambda_j}{R_n^2}$. Then as a simple application of 
Corollary \ref{mild}, we have
\begin{equation}\label{eqn:phiEx1}
\phi(\nu_n,0,\epsilon) \leq C(R_n\epsilon)^{-\frac{1}{\alpha + \delta}}
\end{equation} 
for some $C > 0$.

Moving onto estimating the bias term, assume $u_0 \in \calH_1^{\gamma}(\calC_n)$  and let
$\sum_i \left(\frac{\left\langle u_0,\varphi_i\right\rangle}{i^{\gamma}}\right)^2 = R$.  
The Hilbert scale $\calH_1^{\gamma}(\calC_n)$ of order $\gamma$ is defined with respect to the eigenbasis of the operator $\calC_n$. 
Then, by Corollary \ref{non-cent-small-ball-1}, we have
\begin{equation}\label{eqn:biasEx1}
\mathcal{K}(\nu_n, G(u_0), \epsilon) \leq CR_n^2\left(\epsilon^{-\frac{1
+ 2\delta -2\gamma}{\alpha + \gamma}}\vee1\right).
\end{equation}

Next, to get the optimum $\epsilon_n$, we first minimise
$\phi(\nu_n,G(u_0),\epsilon)$ by fixing the value of $R_n$. However, we also need to take care of the fact that $n\epsilon_n^2 \to \infty$ as $n \to \infty$. Note that
$\phi(\nu_n,0,\epsilon)$ decreases with increasing $R_n$ while
$\mathcal{K}(\nu_n, G(u_0), \epsilon)$ increases with increasing $R_n$. We
choose $R_n$ so that both the terms are of the same order. Setting $R_n =
\epsilon_n^{\beta}$, and for $2\gamma \leq 1 + 2\delta$, and comparing the right hand sides of
\eqref{eqn:phiEx1} and \eqref{eqn:biasEx1}, we must have
$$\beta = \frac{\delta - \gamma}{\alpha + \gamma}$$
 Making the substitution we see that $$\phi(\nu_n,G(u_0),\epsilon_n) \approx
\epsilon_n^{-\frac{1}{\alpha + \gamma}}.$$ Then, to get $\epsilon_n$, we set
 \begin{eqnarray*}
 \phi(\nu_n,G(u_0),\epsilon_n) &=& n\epsilon_n^2\\
 \implies \epsilon_n &=& n^{-\frac{\alpha + \gamma}{1 + 2\alpha + 2\gamma}}.
 \end{eqnarray*}
Similarly, for $\gamma > 1 + 2\delta$, we have $\epsilon_n = n^{-\frac{1 +
2\alpha + 2\delta}{4(1 + \alpha + \delta)}}$. 

We next aim to estimate $k_n$ and $\xi_n$. Following the proof of
equation 5.5 in \cite{ray}, we obtain the expression $$\mathbb{P}\left(\|u -
P_{k_n}u\| \geq \frac{k_n^{-\delta}(1 +
\sqrt{2Ln\epsilon_n^2k_n^{-1}})}{R_n}\right) \leq
\exp^{-Ln\epsilon_n^2}.$$ Substituting $k_n = n\epsilon_n^2$, we have
$$\mathbb{P}\left(\|u - P_{k_n}u\| \geq \frac{L'
k_n^{-\delta}}{R_n}\right) \leq \exp^{-Ln\epsilon_n^2}$$ for some $L'>0$. 
Hence, $\xi_n = \frac{k_n^{-\delta}}{R_n}$ satisfies
the Assumption \ref{assum2''}, and we have $\xi_n = n^{-\frac{\gamma}{1 +
2\gamma + 2\alpha}}$ for $\gamma \leq 1 + 2\delta$ and $\xi_n =
n^{-\frac{2\delta + 1}{4(1 + \alpha + \delta)}}$ for $\gamma > 1 +
2\delta$. 

Given \eqref{eqn:gn-meyer} and expressions for $\xi_n$, $k_n$ and
$\epsilon_n$, it is straightforward to check that Assumption
\ref{assum3''} is satisfied. Hence, $\xi_n$ are the contraction rates.
\end{proof}

\begin{remark}
Note that minimax rates for this problem are known to be 
$\xi_n = n^{-\frac{\gamma}{1 + 2\gamma + 2\alpha}}$ (see \cite{Belister}), 
and our contraction rates match this rate only when $\gamma \leq 1 + 2\delta$.
\end{remark}

\subsubsection{Inversion with linear perturbations}\label{sec:G without SVD}

Continuing withe same theme of model \eqref{M}, we discuss our next class of examples which is based 
around examples discussed in \cite{agapiou} and strictly contains them.

We allow the noise to be colored with covariance $\zeta$. Before we begin the analysis, we shall
note two results from functional analysis.
\begin{theorem}\label{thm:functional1}
Let $K$ and $G_1$ be bounded linear operators defined on a separable Hilbert space, such that they 
are positive and positive definite, respectively. In addition, if $G_1 K$ is assumed to be compact, then $(\I + G_1K)^{-1}$ is continuous.
\end{theorem}

Proof follows by applying Theorem $7$, Section $27$ of \cite{Spectral}, and observing that $G_1K$ cannot have $(-1)$ as an eigenavlue.

\begin{theorem}\label{thm:functional2}
Let $G_1$ or $G_1^{-1}$ be positive definite, compact and self adjoint. Then, $\|G_1^p(x)\|^q \geq \|G_1^q(x)\|^p$ for $p \geq q > 0$ and $\|x\| = 1$. 
\end{theorem}
The result can be readily proved by applying Jensen's inequality.

For the examples to be discussed in this section, we set $\calH_1 = \calH_2 = \calH$. 
Let $G$ be compact, self adjoint and positive definite.
Let the covariance operator of noise be $\zeta \equiv \left(G^{-r} +
K_1\right)^{-2}$ for $ r\in (0,1)$, and $K_1$ be a continuous,
self-adjoint and positive operator on $\calH$. In case where $r = 0$, we take $K_1 = 0$, which
corresponds to the case of white noise.
We shall pick Gaussian priors with covariance operators
$\mathcal{C}_n \equiv \frac{\left(G^{-t} + K_2\right)^{-l}}{R_n^2}$ for $t > (1-r)$ and $l\leq 2$. 
We choose $l$ and $t$ such that $G^{lt}$ is trace class. We will show below that $\mathcal{C}_n$ is a viable covariance operator. 
\begin{lemma}
$\mathcal{C}_n$, defined above, is a self adjoint, positive definite and trace class operator.
\end{lemma}
\begin{proof}
Since $G^{-t}$ is positive definite and $K_2$ is positive, it then follows that $\mathcal{C}_n$ is positive definite. Further, $\mathcal{C}_n$ is self adjoint since $G^{-t}$ and $K_2$ are self adjoint. We now need to show that $\mathcal{C}_n$ or equivalently $\left(G^{-t} + K_2\right)^{-l}$ is trace class. We have, $$\left(G^{-t} + K_2\right)^{-l} = \left((\I + G^{t}K_2)^{-1}G^t\right)^l.$$
Since $\left(\mathcal{I} + G^{t}K_2\right)^{-1}$ is continuous by Theorem \ref{thm:functional1}, $\left((\I + G^{t}K_2)^{-1}G^t\right)^l$ is compact.

Let $\{x_i\}$ be the eigenvectors of $G$. It will suffice to show that $$\sum \left\langle\left((\I + G^{t}K_2)^{-1}G^t\right)^l x_i,x_i\right\rangle < \infty.$$
Since $\left((\I + G^{t}K_2)^{-1}G^t\right)^l$ is compact, self adjoint and positive definite, we have $$\sum \left\langle\left((\I + G^{t}K_2)^{-1}G^t\right)^l x_i,x_i\right\rangle = \sum \|\left((\I + G^{t}K_2)^{-1}G^t\right)^{\frac{l}{2}} x_i\|^2.$$
Further, by Theorem \ref{thm:functional2}, we have $$\sum \|\left((\I + G^{t}K_2)^{-1}G^t\right)^{\frac{l}{2}} x_i\|^2 \leq \sum \|\left((\I + G^{t}K_2)^{-1}G^t\right) x_i\|^l.$$
Let $\rho_i$ be the eigenvalues of $G$ and $\|(\I + G^{t}K_2)^{-1}\| = R_3$. Then we have $$\sum \|\left((\I + G^{t}K_2)^{-1}G^t\right) x_i\|^l \leq R_3^l\sum\rho_i^{lt} < \infty$$
which follows as consequence of our assumption of $G^{lt}$ being trace class.
\end{proof}

This makes $\mathcal{C}_n$ a viable covariance operator for a Gaussian prior supported on $\mathcal{H}$
Further, assume that there exist positive constants 
$\alpha$ and $\delta$ such that  $\rho_j^{1 -r} = j^{-\alpha}$ and $\rho_j^{lt} = j ^{-1 - 2\delta}$.

\begin{theorem}
The model described above is well-posed. Further, the contraction rates $\xi_n$ for $u_0 \in \calH_1^{\gamma}(\calC_n)$ are given by the following expressions. 
\begin{equation}
\xi_n = \left\{ 
\begin{array}{lc} 
n^{-\frac{\gamma'}{1 + 2\alpha + 2\gamma'}} & 2\gamma' \leq   1+ 2\delta, R_n = n^{\frac{\gamma - \delta}{1 + 2\alpha + 2\gamma}}\\
n^{-\frac{2\delta + 1}{4(1 + \alpha + \delta)}} & 2\gamma' > 1 + 2\delta, R_n = \frac{1}{4(1 + \alpha + \delta)}
\end{array}
\right.
\end{equation}
where $\gamma' = \frac{\gamma(1 - r)}{t}$.
\end{theorem}
\begin{proof}We shall begin with estimating $\mu_n \{ u : \left\|G(u) -G(u_0)\right\|_{\zeta} \leq \epsilon \}$. Note that 
$\zeta^{-1/2}G = (\I + K_1G^r)G^{-r + 1}$. Since  $r < 1$, this implies that $\zeta^{-\frac{1}{2}}G$ is continuous, then
well-posedness of the posterior follows by theorem \ref{thm:well-posedness}.

Next, writing $\|(\I + K_1G^r)\| = R_1$ observe that

\begin{eqnarray*}
\|G^{-r + 1} (u - u_0)\| \leq \epsilon &\implies& \|\left(\I +
K_1G^r\right)G^{-r + 1} (u - u_0)\| \leq R_1\epsilon\\
\implies \{u : \|G^{-r + 1} (u - u_0)\| \leq \epsilon\} &\subset& \{u : \|\left(\I + K_1G^r\right)G^{-r + 1} (u - u_0)\| \leq R_1\epsilon\}.
\end{eqnarray*}

Hence a lower bound for $\mu_n \{ u : \left\|G^{-r+1}(u-u_0)\right\| \leq
K\epsilon \}$ will also be a lower bound for $\mu_n \{ u : \left\|G(u) -
G(u_0)\right\|_{\zeta} \leq \epsilon \}$.
Therefore, it suffices to find a sequence $\epsilon_n$ for the operator
$G^{-r + 1}$ instead of the operator $\zeta^{-1/2}G$. The computations
involved in finding such $\{\epsilon_n\}$
will depend on the choice of our prior $\mu_n$. Assume that $\mathcal{C}_n$ has eigenbasis $\{\varphi_j\}$ with 
$\{\frac{\lambda_j^2}{R_n^2}\}$ the corresponding eigenvalues.

As before, let $\nu_n$ be the push forward of the prior $\mu_n$ with respect to the operator $G^{-r + 1}$. Let the 
eigenvalues of the covariance operator of $\nu_n$, given by $G^{-r + 1}\mathcal{C}_nG^{-r + 1}$, be $\{\alpha_{j,n}\}$. 
Let the eigenpair of $G$ be $\{e_j, \rho_j\}$. Further, let $D^*_j$ be the span of
$\{e_j,e_{j + 1}......\}$. Then, by the second part of Theorem \ref{thm:minimax}, we have
\begin{align*}
\alpha_{j,n} &= \frac{1}{R_n^2} \min_{\{S : \textnormal{codim} (S) = j - 1\}}\max_{\{x\in S : \|x\| = 1\}}
\langle G^{-r + 1}\left(G^{-t} + K_2\right)^{-l}G^{-r + 1}x,x\rangle\\
&\leq  \frac{1}{R_n^2}\max_{\{x : \|x\| = 1, x \in D^*_j\}}\frac{\left\langle \left(G^{-t} +
K_2\right)^{-l}G^{-r + 1}x,G^{-r + 1}x\right\rangle}{\langle G^{-r + 1}x,G^{-r +
1}x\rangle} \max_{\{x : \|x\| = 1, x \in D^*_j\}} \langle G^{-2r +
2}x,x\rangle\\
&= \frac{\rho_j^{2 -2r}}{R_n^2} \max_{\{x : \|x\| = 1, x \in D^*_j\}}\left\langle
\left(\left(G^{-t} + K_2\right)^{-1}\right)^{l}x,x\right\rangle\\
&= \frac{\rho_j^{2 -2r}}{R_n^2} \max_{\{x : \|x\| = 1, x \in D^*_j\}}
\left\|\left(\left(G^{-t} + K_2\right)^{-1}\right)^{l/2}x \right\|^2\\
&\leq  \frac{\rho_j^{2 -2r}}{R_n^2} \max_{\{x : \|x\| = 1, x \in D^*_j\}}
\left\|\left(\left(G^{-t} + K_2\right)^{-1}\right)x\right\|^l \ \ \ \
\,\,\,\text{(by Theorem \ref{thm:functional2} and $l\le 2$)}\\
&= \frac{\rho_j^{2 -2r}}{R_n^2} \max_{\{x : \|x\| = 1, x \in D^*_j\}} \left\|\left(\I +
G^tK_2\right)^{-1}G^tx \right\|^l\\
&\leq  \frac{\rho_j^{2 -2r}}{R_n^2} \left\|\left(\I + G^tK_2\right)^{-1}\right\|^{l}
\left\|\left.G^t\right|_{D^*_j}\right\|^{l}\\
&\leq  C_1\frac{\rho_j^{2 -2r + lt}}{R_n^2}.
\end{align*}
Similarly, applying the fourth part of Theorem \ref{thm:minimax}, we see
that $\alpha_{j,n} \geq C_2\rho_j^{2 - 2r + lt}$. Then we have $\alpha_{j,n} \approx \frac{j^{-(1 + 2\alpha + 2\delta)}}{R_n^2}$.

Applying Corollary \ref{mild}, we have $$\phi(\nu_n,0,\epsilon) \geq \exp(-(C_3 R_n \epsilon)^{-\frac{1}{\alpha+ \delta}}).$$
Let us take $u_0 \in \calH^{\gamma}(\mathcal{C}_n)$ and let 
$\sum_i \left(\frac{\left\langle u_0,\varphi_i\right\rangle}{i^{\gamma}}\right)^2 = R$. 
Setting $\gamma' = \frac{\gamma(1 - r)}{t}$, by Theorem \ref{non-cent-small-ball-2}, 
$$\mathcal{K}(\nu_n, u_0, \epsilon) \leq KR_n^2\left(\epsilon^{-\frac{1 + 2\delta - 2\gamma'}{\alpha + \gamma'}}\vee 1\right).$$ 
From here on, we can repeat the calculations of the last example and obtain $\xi_n$ and $k_n$ corresponding to $\gamma'$ instead of $\gamma$. 
Next, we need to check if the sequences $\xi_n$, $k_n$ and $\epsilon_n$ satisfy Assumption \ref{assum3''}, ensuring that the sequence $\xi_n$ 
indeed is the rate of contraction, though a suboptimal one due to the smaller exponent $\gamma'$.

We need to show that $g_{k_n} \leq
C_1(\frac{\xi_n}{\epsilon_n})^2$. For this, we
estimate $g_{n}$. Defining $S_1 = \|(\I + G^rK_1)^{-1}\|$ and $S_2 = \|(\I -
\mathcal{C}_n^{\frac1{l}}K_2)\|$, using Theorem \ref{thm:functional2}, 
for $h \in \text{span} \{\varphi_1,..,\varphi_n\}$ with $\|h\|=1$,
we have the following\\

Observe, first that we can write 
$$ \|\zeta^{1/2}G^{-1}h\| = \|(\I + G^rK_1)^{-1}G^{r - 1}h\| =  \|(\I +
G^rK_1)^{-1}(\mathcal{C}_n^{-1/l} - K_2)^{\frac{1 - r}{t}}h\|$$
Then, setting $S_1 = \|(\I + G^rK_1)^{-1}\|$ and $S_2 = \|(\I -
\mathcal{C}_n^{\frac1{l}}K_2)\|$ and using Theorem \ref{thm:functional2}, 
for $h \in \text{span} \{\varphi_1,..,\varphi_n\}$ with $\|h\|=1$
we have 
$$ \|\zeta^{1/2}G^{-1}h\| \leq S_1 S_2^{\frac{1 - r}{t}} \|\mathcal{C}_n^{-1/l}h\|^{\frac{1 - r}{t}} \leq 
S_1 S_2^{\frac{1 - r}{t}}\rho_n^{r - 1}.
$$
Now recalling the definition of $g_n$ from \eqref{eqn:def-gr1}, it follows that Assumption \ref{assum3''} is indeed satisfied. Lastly, the assumption stated in \eqref{eqn:u0-assum} 
can easily be checked making $\xi_n$ the contraction rates.
\end{proof} 

\begin{remark}
Note that we can execute the above computations for Example \ref{ex:mild} as well. Also notice that in this case, we shall get $\gamma' = \gamma$
and hence will get the same rate of contraction as in the conjugate case
with prior $\calC$ and linear operator $G = \calC^{l-r}$. We skip the details here since we will be doing very similar calcuations in the subsection below on severely ill posed problems. We also note here that the above examples achieve the third part in Subsection \ref{contribution}.
\end{remark}

\subsubsection{Contraction rate for severely ill posed problems}\label{severely ill posed}
In this subsection, we shall discuss well-posedness and contraction rates for severely ill posed problems. We now describe the class of examples we deal with here. 

We have $\calH_1 = \calH_2 = \calH$. We define the relevant operators in terms of the covariance of the prior $\mu_n$, denoted by $\mathcal{C}_n = \frac{\mathcal{C}}{R_n^2}$ with eigenpair $\{\phi_j, \frac{\lambda_j}{R_n^2}\}$.
Let the covariance operator of noise be $\zeta \equiv \left(\mathcal{C}^{-r} +
K_1\right)^{-2}$ for $ r \geq 0$, and $K_1$ be a continuous,
self-adjoint and positive operator on $\calH$. In case where $r = 0$, we take $K_1 = 0$. Let the operator to be inverted be $G \equiv (\exp (\mathcal{C}^{-\beta'}) + K_2)^{-1}$ for $\beta'>0$ and $K_2$ a continuous,
self-adjoint and positive operator on $\calH$. We know that $\exp (\mathcal{C}^{-\beta'})$ is well defined with eigenpair $\{\phi_j, \exp (\lambda_j^{-\beta'})\}$. 
We set $\beta,\delta>0$ such that $\lambda_j = j^{-1 - 2\delta}$ and also $\lambda_j^{-\beta'} = j^{\beta}$.

\begin{theorem}
The model described above is well-posed. Further, the contraction rate $\xi_n$ for $u_0 \in \calH^{\gamma}(\mathcal{C}_n)$ is given by 
$$\xi_n = (\log n)^{-\frac{\gamma}{\beta}}$$ with $R_n = n^{\frac{1}{2} - \sigma}$ for any $0 < \sigma < \frac{1}{2}$.
\end{theorem}
\begin{proof} We see that $G$ and $\zeta$ are positive definite, compact and self adjoint, hence $G$ has an eigenbasis and $\zeta$ is a viable covariance operator for a noise measure. 
We prove the statement for $G$, the result for $\zeta$ can be proved in a similar fashion. $G$ is clearly positive and self adjoint. 
Further, we have $$G = (\mathcal{I} + \exp (-\mathcal{C}^{-\beta'})K_2)^{-1}\exp (-\mathcal{C}^{-\beta'}).$$ 
From Theorem \ref{thm:functional1}, it follows that $G$ is compact.

We shall begin by estimating $\mu_n \{ u : \left\|G(u) -
G(u_0)\right\|_{\zeta} \leq \epsilon \}$. We have $$\zeta^{-1/2}G = (\I +
K_1\mathcal{C}^r)(\I + \mathcal{C}^{-r}\exp(-\mathcal{C}^{-\beta'})K_2\mathcal{C}^r)^{-1}\mathcal{C}^{-r}\exp (-\mathcal{C}^{-\beta'}).$$ 
Note that $T\equiv (\I + \mathcal{C}^{-r}\exp(-\mathcal{C}^{-\beta'})K_2\mathcal{C}^r)^{-1}$ is continuous by Theorem \ref{thm:functional1}.
Writing $\|(\I + K_1\mathcal{C}^r)T\| = R_1$, we have

\begin{eqnarray*}
\|\mathcal{C}^{-r}\exp (-\mathcal{C}^{-\beta'}) (u - u_0)\| \leq \epsilon &\implies& \|(\I +
K_1\mathcal{C}^r)T \mathcal{C}^{-r}\exp (-\mathcal{C}^{-\beta'}) (u - u_0)\| \leq R_1\epsilon\\
\implies \{u : \| \mathcal{C}^{-r}\exp (-\mathcal{C}^{-\beta'})(u - u_0)\| \leq \epsilon\} &\subset& \{u : \|\zeta^{-1/2}G (u - u_0)\| \leq R_1\epsilon\}.
\end{eqnarray*}
Hence a lower bound for $\mu_n \{ u : \left\|\mathcal{C}^{-r}\exp (-\mathcal{C}^{-\beta'})(u-u_0)\right\| \leq
K\epsilon \}$ will also be a lower bound for $\mu_n \{ u : \left\|G(u) -
G(u_0)\right\|_{\zeta} \leq \epsilon \}$. The calculations also imply that $\zeta^{-\frac{1}{2}}G$ is continuous 
and hence the posterior is well-posed by Theorem \ref{thm:well-posedness}. It will thus suffice to calculate 
$\epsilon_n$ for the operator $\mathcal{C}^{-r}\exp (-\mathcal{C}^{-\beta'})$. The pushforward $\nu_n$ of 
$\mu_n$ under the operator $\mathcal{C}^{-r}\exp (-\mathcal{C}^{-\beta'})$ has eigenpair 
$\{\phi_j, \lambda_j^{1 - 2r}\exp(-2\lambda_j^{-\beta'})\}$. 

Applying Corollary \ref{severe}, we have $$\phi(\nu_n,0,\epsilon) \leq C_5 (- \log(\epsilon R_n))^{\frac{\beta + 1}{\beta}}.$$
By Theorem \ref{non-cent-small-ball-3}, we have $$\mathcal{K}(\nu_n, G(u_0), \epsilon) \leq R R_n^2
\left((-\log (\epsilon))^{\frac{-2\gamma + 1 + 2\delta}{\beta}}\vee 1\right).$$
By setting $R_n = n^{\frac{1}{2} - \sigma}$ for $0 < \sigma < \frac{1}{2}$, we notice that the bias term $\mathcal{K}(\nu_n, G(u_0), \epsilon)$ dominates the expression for
$\phi(\nu_n,G(u_0),\epsilon)$, implying that we need $\epsilon_n$ such that $$k n\epsilon_n^2 \leq n^{1 - 2\sigma}
\left((-\log (\epsilon_n))^{\frac{-2\gamma + 1 + 2\delta}{\beta}}\right)$$ for some $k > 0$. 

It is not difficult to see that $\epsilon_n = \frac{(\log n)^{\frac{-2\gamma + 1 + 2\delta}{2\beta}}}{n^{\sigma}}$ satisfies the above inequality. 
Next, we choose $k_n \leq k_2(\log n)^{\frac{1}{\beta}} < n\epsilon_n^2$ for some $k_2$ to be chosen later.
As in the example discussed in Section \ref{sec:meyer}, we have $$\mathbb{P}\left(\|u -
P_{k_n}u\| \geq \frac{k_n^{-\delta}(1 +
\sqrt{2Ln\epsilon_n^2k_n^{-1}})}{R_n}\right) \leq
\exp^{-Ln\epsilon_n^2}.$$
Putting in the expressions for $k_n$ and $\epsilon_n$, we have $$\mathbb{P}\left(\|u -
P_{k_n}u\| \geq C(\log n)^{\frac{-\gamma}{\beta}}\right) \leq
\exp^{-Ln\epsilon_n^2}$$ for some $C > 0$.
Hence, $\xi_n = (\log n)^{\frac{-\gamma}{\beta}}$ is our candidate for rate of contraction. Given the expressions of $\xi_n,\epsilon_n,k_n$, we need to estimate $g_k$ to check if condition \ref{assum3''} is satisfied.

Noting that $\zeta^{\frac{1}{2}}$ is continuous, we have for $h \in span\{\phi_1,....\phi_k\}$ and some $k,k_1 > 0$,
\begin{eqnarray*}
\|\zeta^{\frac{1}{2}}G^{-1}h\| &\leq& k \|\left(\exp(\mathcal{C}^{-\beta'}) + K_2\right)h\|\\ &\leq& k_1 \exp(k^{\beta}).
\end{eqnarray*}
Thus, $g_{k_n} = n^{k_2^{\beta}}$. Hence, condition \ref{assum3''} is satisfied for 
$k_2 < \sigma^{\frac{1}{\beta}}$. Further, it is easy to check that the true solution $u_0$ satisfies the condition
stated in \eqref{eqn:u0-assum} implying that $\xi_n = (\log n)^{\frac{-\gamma}{\beta}}$ is the contraction rate.
\end{proof}

\begin{remark}
 The rate can be reproduced in the case when the prior and operator basis are related as in the first example using similar calculations. 
 Note that this is also the minimax rate and has been achieved using scalable priors which do not depend on the true solution. 
 Similar results have been obtained by \cite{Knapik} and \cite{Zhang} for conjugate case. This achieves the fourth part in Subsection \ref{contribution}.
\end{remark}

\subsubsection{Inverting Semilinear operator}\label{sec:semilinear}

The model  here is
\begin{equation}\label{semi}
    y = G^{*}(u) + \frac{1}{\sqrt{n}} \eta.
\end{equation}
Here, $G^{*}$ is a one-one semilinear
and continuous map $G^{*}: \mathcal{H}_1 \to \mathcal{H}_2$.
We rewrite the above as
\begin{equation*}
    y = G(u^{*}) + \frac{1}{\sqrt{n}}\eta,
\end{equation*}
where $u^{*} = G^{-1}G^{*}(u)$ where $G$ is an injective, compact and linear map. 
We assume that $u^{*}$ is an bijective
continuous transform with continuous inverse. We are now in a position to use Lemma \ref{corollary1}.

 To restate, we need
following conditions to be satisfied to get contraction rate for the model \eqref{semi}
 \begin{eqnarray*}
 \mu_n\{u : \|P_{k_n}(u^{*}) - u^{*}\| \geq C_2\xi_n\} &\leq& \exp(-(C +
4)n\epsilon_n^2)\\
    \mu_n\{u : \|G(u^{*}) - G(u_0^{*})\| \leq k\epsilon_n\} &\geq&
\exp(-Cn\epsilon_n^2)\\
    \|P_{k_n}(u_0^{*}) - u_0^{*}\| \leq \xi_n.
 \end{eqnarray*}
Since by the continuity properties of the transformation, we have $\{u :
\|u - u_0\| \leq k_1\epsilon\} \subset \{u : \|u^{*} - u_0^{*}\| \leq
\epsilon\} \subset \{u : \|u - u_0\| \leq k_2\epsilon\}$, we can replace
the second condition by
\begin{equation*}
  \mu_n\{u : \|G(u) - G(u_0)\| \leq k\epsilon_n\} \geq
\exp(-Cn\epsilon_n^2).
\end{equation*}

We now consider a specific example of $G^{*}$ and prior for which we can
verify the above conditions. Let $\Delta$ be the Laplacian and $G^{*} =
(\Delta + k \sin)^{-1}$ with $k < \frac{1}{8}$. Note that $G^{*}$ is the
sine-Gordon equation for $k =1$. Next, writing $\{e_i\}$ for eigenvectors of the Laplacian (on an appropriate bounded interval)
we define $\mathcal{H}$ as
\begin{equation*}
\mathcal{H} = \{u : u = \sum a_i e_i : \sum (i a_i)^2 < \infty\}.
\end{equation*}
We take the covariance of the prior to be $\mathcal{C} = \Delta^{-\delta'}$ 
with $\delta' > \frac{3}{2}$ so as to ensure that the prior is well supported on $\mathcal{H}$. 
We set $G = \Delta^{-1}$, and observe that under the norm of $\mathcal{H}$, the smoothness of $G$ is 
$\alpha = 1$ and that of $\mathcal{C}$ is $2\delta' -2 \equiv 1 + 2\delta$. 

Finally, it is not difficult to notice that the map $G^{-1}G^{*} =
(\mathcal{I} + k \sin(\Delta^{-1}))^{-1}$,
is a continuous bijection with continuous inverse.

\begin{theorem}
The contraction rate $\xi_n$ for the above model and true solution $u_0 \in \mathcal{H}^{\gamma}(\mathcal{C})$is given by 
$$\xi_n = n^{-\frac{\gamma\wedge \delta}{2\alpha + 2\delta + 1}}.$$ 

\end{theorem}
We will
now use Theorem 3.2.1 of \cite{Transformation} to
check the first condition. In this example, we have $u = k \sin
(\Delta^{-1})$. Due to conflict of notation, we will rather use $w = k
\sin (\Delta^{-1})$. It is easy to check that  $\nabla w(u) = k \cos u
(\Delta^{-1})$  and that it satisfies $\|\nabla w\| \leq k < 1$. Further,
since the operator ($k \cos$) is bounded on $\mathcal{H}$ and the inverse of Laplacian
has bounded Hilbert-schmidt norm, $w$ satisfies conditions in Theorem 3.2.1 of \cite{Transformation},
thus
\begin{eqnarray*}
&& \mu_n\{u : \|P_{k_n}(u^{*}) - u^{*}\| \geq C_2\xi_n\} \\
&=& \int_{\mathcal{H}} \mathcal{I}_{\{u : \|P_{k_n}(u) - u\| \geq C_2\xi_n\}}
\left|\det_2 (\mathcal{I}_{cm} + \nabla w)\right| \exp \left(-\delta(w) -
\frac{|w|_{cm}^2}{2}\right) d\mu(u).
\end{eqnarray*}
$\det_2$ is the Carleman-Fredholm determinant and the subscript $cm$
denotes the Cameron martin space/norm. Hence $\det_2(\mathcal{I}_{cm} +
\nabla w)$ is the Carleman-Fredholm determinant of the operator
$\mathcal{I}_{cm} + \nabla w$ on the Cameron-Martin space of the measure
$\mu_n$. The determinant exists and is uniformly bounded for all $u$ since
$\Delta^{-1}$ is Hilbert-Schmidt and $\cos u$ is uniformly bounded on $u$.
By Cauchy-Schwartz inequality, we have
\begin{eqnarray*}
&& \int_{\mathcal{H}} \mathcal{I}_{\{u : \|P_{k_n}(u) - u\| \geq C_2\xi_n\}}
\left|\det_2(\mathcal{I}_{cm} + \nabla w)\right| \exp \left(-\delta(w) -
\frac{|w|_{cm}^2}{2}\right) du\\
& \leq & \left(\int_{\mathcal{H}} \mathcal{I}_{\{u : \|P_{k_n}(u) - u\|
\geq C_2\xi_n\}} \left|\det_2(\mathcal{I}_{cm} + \nabla w)\right|^2
du\right)^{\frac{1}{2}}\left(\int_{\mathcal{H}} \exp \left(-2\delta(w) -
|w|_{cm}^2\right) du \right)^{\frac{1}{2}}\\
& \leq & R \left(\mu_n\{u : \|P_{k_n}(u) - u\| \geq
C_2\xi_n\}\right)^{\frac{1}{2}}\left(\int_{\mathcal{H}} \exp
\left(-2\delta(w) - |w|_{cm}^2\right) du \right)^{\frac{1}{2}}\\
& \leq & R \left(\mu_n\{u : \|P_{k_n}(u) - u\| \geq
C_2\xi_n\}\right)^{\frac{1}{2}}\left(\int_{\mathcal{H}} \exp
\left(2|\delta(w)|\right) du \right)^{\frac{1}{2}}.
\end{eqnarray*}
where $R$ is upper bound on the CarlemanFredholm determinant.
We can estimate the first half of last expression as before. For the
second half, we use Proposition B.8.1 from  \cite{Transformation}. Since Hilbert-Schmidt norm of $\nabla w$ is uniformly
bounded, the condition in Proposition B.8.1 is satisfied. Thus, we have
\begin{equation*}
   \mu_n\{u : \|P_{k_n}(u^{*}) - u^{*}\| \geq C_2\xi_n\} \leq R' R
\left(\mu_n\{u : \|P_{k_n}(u) - u\| \geq
C_2\xi_n\}\right)^{\frac{1}{2}}.
\end{equation*}
Now, with the replaced second condition, the conditions are reduced to the
conjugate case and we can get the appropriate contraction rates.

\subsection{Compactly supported prior}
Our model in this example will be similar to previous ones. Let $\{e_i\}$
be the basis of the prior and let $G'$ be a linear operator such that
$\{(\rho_i, e_i)\}$ be its eigenpair. We assume our linear operator to be
$G = (G'^{-1} + W)^{-1} = (\mathcal{I} + G'W)^{-1}G' = KG'$. Under the
compactly supported prior $\mu$, $u \in \mathcal{H}_1$ is given by
\begin{equation*}
u = \sum_k k^{-(\delta + \frac{1}{2})}u_k e_k.
\end{equation*}
Where $u_k$ are random variables supported on the interval $[-B,B]$. Further, we define Sobolev balls $\mathcal{H}^{\gamma}$ on the basis $\{e_i\}$.
\begin{theorem}
The contraction rate $\xi_n$ for the above model and true solution $u_0 \in \mathcal{H}^{\gamma}$is given by 
$$\xi_n = n^{\frac{-\delta}{2\alpha + 2\delta + 1}}$$ when $\delta < \gamma$.
\end{theorem}
As in the last example, it can be shown that the model satisfies
Assumption \ref{assum3''}. We need to now verify the Assumptions \ref{assum1} and \ref{assum2''}. Note that
\begin{eqnarray*}
\mu\left\{\left\|G(u - u_0)\right\| \leq \epsilon\right\} \geq  \mu\left\{\left\|G'(u -
u_0)\right\| \leq \frac{\epsilon}{\left\|K\right\|}\right\}.
\end{eqnarray*}

Since $\|K\|$ is a finite quantity, we can as well focus on $\mu\{\left\|G'(u - u_0)\right\| \leq \epsilon\}$.
Assume $u_0 = \sum k^{-(\gamma + \frac{1}{2})}b_k e_k$  with $\gamma \geq
\delta$. Note that $b_k \to 0$.
\begin{eqnarray*}
\mu\{\left\|G'(u - u_0)\right\| \leq \epsilon\} = \mu\left\{\sum k^{-2(\alpha +
\gamma + \frac{1}{2})}(k^{-(\delta - \gamma)}u_k - b_k)^2 \leq
\epsilon^2\right\}.
\end{eqnarray*}
Pick J such that $B^2J^{-2(\alpha + \delta)} \approx \epsilon^2$. Then we have
\begin{eqnarray*}
&& \mu\left\{\sum k^{-2(\alpha + \gamma + \frac{1}{2})}(k^{-(\delta - \gamma)}u_k
- b_k)^2 \leq \epsilon^2\right\} \\
&=& \mu\left\{\sum_1^{J} k^{-2(\alpha + \gamma +
\frac{1}{2})}(k^{-(\delta - \gamma)}u_k - b_k)^2 \leq \epsilon^2 - \sum_{J
+ 1}^{\infty} k^{-2(\alpha + \gamma + \frac{1}{2})}(k^{-(\delta -
\gamma)}u_k - b_k)^2\right\} \\
& \geq &  \mu\left\{\sum_1^{J} k^{-2(\alpha + \gamma +
\frac{1}{2})}(k^{-(\delta - \gamma)}u_k - b_k)^2 \leq c_1\epsilon^2\right\} \\
&\geq & \mu\left\{\max_{1 \leq k \leq J} k^{-2(\alpha + \gamma +
\frac{1}{2})}(k^{-(\delta - \gamma)}u_k - b_k)^2 \leq c_1\frac{\epsilon^2}{J}\right\} \\
&\geq & \prod_1^J \mu\left\{(k^{-(\delta - \gamma)}u_k - b_k)^2 \leq c_1\frac{\epsilon^2}{J}\right\}.
\end{eqnarray*}
We can see that
\begin{equation*}
\mu\{(k^{-(\delta - \gamma)}u_k - b_k)^2 \leq c_1\epsilon^2\} \geq
\frac{c_2\epsilon}{BJk^{\gamma - \delta}}.
\end{equation*}
Hence, we have
\begin{eqnarray*}
\mu\{\left\|G(u - u_0)\right\| \leq \epsilon\} \geq \prod_1^J
\mu\{(k^{-(\delta - \gamma)}u_k - b_k)^2 \leq c_1\epsilon^2\}\\ \geq \left(\frac{c_2\epsilon}{BJ}\right)^J (J!)^{\delta - \gamma} \\ \approx
\exp(-c_3 J \ln J) \approx \exp(-c_4\epsilon^{\frac{-1}{\alpha +
\delta}}\ln \epsilon).
\end{eqnarray*}
Thus, Assumption \ref{assum1} is satisfied with $\epsilon_n = (\frac{\ln
n}{n})^{\frac{\alpha + \delta}{2\alpha +2\delta + 1}}$.
Also, note that we have no convergence for $\gamma < \delta$ since the
prior does not assign any mass to sufficiently small balls around the true
solution. Next, we check Assumption \ref{assum2''}. Let $k_n \approx n\epsilon_n^2$,
then for $\xi_n = n^{\frac{-\delta}{2\alpha + 2\delta + 1}}$
\begin{eqnarray*}
\left\|u - P_{k_n}u\right\|^2 = \sum_{k_n}^{\infty} k^{-(2\delta +
1)}u_k^2 \leq c_4 B^2 k_n^{-2\delta} \leq c_5 \xi^2.
\end{eqnarray*}

Thus, Assumption \ref{assum2''} is satisfied. This gives the contraction rate to be
$\xi_n = n^{\frac{-\delta}{2\alpha + 2\delta + 1}}$.
This achieves the fifth part in Subsection \ref{contribution}.

\section{minimax rates}
We will now discuss the appropriate minimax rates for inverse problems with non conjugate priors. 
The usual minimax rates are calculated for true solutions lying in Sobolev
balls (ellipsoids) corresponding to the eigenbasis of the operator. Hence,
in the conjugate case, it makes sense to compare the convergence rates on
Sobolev balls to minimax rates since the Sobolev balls in both cases are
the same. In our case, however, we get rates for true solutions lying in
Sobolev balls corresponding to the prior basis which is different from the
operator basis. The Sobolev balls corresponding to the prior basis may be different as well.
 We show below that for Example \ref{ex:mild}, arbitrarily smooth Sobolev balls in prior basis may not belong to Sobolev ball of a fixed order in operator basis.

Let $\mathcal{H}$ be $\mathbb{L}^2_{per}[0,1]$, i.e the set of periodic
square integrable functions on $[0,1]$. Let operator $G$ in our model be
defined on $\mathcal{H}$ by $G = (\Delta + W)^{-1}$ where $\Delta$ is the
Laplace operator (acts on the elements via Fourier series) and $W$ is the
operator denoting multiplication by a continuous positive function $w$.
Further, let the covariance operator $\mathcal{K}$ of the prior be
$\Delta^{-1}$. Note that for $h \in \mathcal{H}$, we have
$G^{-1}\mathcal{K}h \in \mathcal{H}$. If the Sobolev spaces of second order smoothness corresponding
to the respective operators are the same, we must have
$G^{-2}\mathcal{K}^2h \in \mathcal{H}$. Simplifying, we get the equivalent
condition $\Delta w \Delta^{-2}h \in \mathcal{H}$. It is clearly false for
$h = \sin x$ and $w = 10 + \sum \frac{\sin 3kx}{k^2}$. Hence, the Sobolev
spaces of the true solutions in our case is clearly different from the
standard Sobolev spaces of the minimax rates. It is easy to apply the same argument to show that Sobolev balls of arbitrarily hgh smoothness in prior basis will not belong to second order Sobolev balls of Operator basis. Fortunately, it is elementary to calculate the minimax rates for the Sobolev balls in the examples we have used by the method of Belitser et. al. \cite{Belister}. We outline the idea below.
Define $y_k = \left\langle y,\tilde{\phi_k}\right\rangle = \left\langle
u,\phi_k\right\rangle + \frac{1}{\sqrt{n}}\left\langle
z,\tilde{\phi_k}\right\rangle$.  Note that $\tilde{\phi_k}$ is just $G^{-1}\phi$ and $\left\langle z,\tilde{\phi_k}\right\rangle$ are not independent. We
follow Theorem 1 of Belitser et.al. \cite{Belister} 
and let our estimator for $u_k =
\left\langle u,\phi_k\right\rangle$ be $x_ky_k$ for some $x_k$ to be
decided. The minimax risk is
\begin{equation*}
  r_n = \sum (1 -x_k)^2u_{0,k}^2 + \frac{x_k^2\|\tilde{\phi}_k\|^2}{n}
\end{equation*}
In the example we have used, $\|\tilde{\phi}_k\|^2$ serves the role of
$\sigma^2$ in \cite{Belister}. The rest of the proof is exactly
the same giving us the same minimax rates.

We also note below that we cannot obtain a uniform rate of convergence for priors of fixed smoothness (decay rate of eigenvalues) but varying eigenbasis and hence, it needs to be dealt on a case by case basis. 

Assume the model is $$y = G(u) + \frac{1}{\sqrt{n}}\eta$$
 where $u \in \mathcal{H}_1$ and $G : \mathcal{H}_1 \to \mathcal{H}_2$ is a compact operator such that $G^TG$ has eigenpair $\{e_i,\rho_i^2\}$. $\eta$ is the white noise on $\mathcal{H}_2$. The prior on $\mathcal{H}_1$ is a centered Gaussian measure $\mu_{\phi}$ supported on $\mathcal{H}_1$ with covariance operator given by $\mathcal{C}_{\phi}$. $\phi$ denotes the eigenbasis $\{\phi_i\}$ of $\mathcal{C}_{\phi}$. Assume that the eigenvalue of $\mathcal{C}_{\phi}$ corresponding to $\phi_i$ is $\lambda_i$.
\begin{claim}
Given any sequence $\xi_n \to 0$, we can find a $\{\phi_i\}$ such that the contraction rate of the above model is slower than $\xi_n$.
\end{claim}
\begin{proof}

Consider a sequence $\theta_i$ such that $\theta_i \to 0$ with $\theta_i > 0$ and $\theta_i \neq \theta_j$ when $i \neq j$. We shall now construct $\{\phi_i\}$ as a permutation of the basis $\{e_i\}$, that is $\phi_i = e_{\sigma(i)}$ for some permutation $\sigma : \mathbb{N} \to \mathbb{N}$. For $i = 2r - 1$, we put $\sigma(i) = 2k - 1$ where $k$ is the smallest integer such that $\rho_{2k - 1} \leq \theta_i$ and $2k - 1 \neq \sigma(j)$ for all $j < i$. We have thus injectively mapped all the odd integers into the odd integers . We can now map the even integers on the remaining natural numbers bijectively. Rewriting the original model in $\{\phi_i\}$ basis, we have the model $$y = G_1(u) + \frac{1}{\sqrt{n}}\eta$$ $u \in \mathcal{H}_1$ and $G_1 : \mathcal{H}_1 \to \mathcal{H}_2$ is a compact operator with eigenpair $\{\phi_i,\rho_{\sigma_i}\}$. The prior has covariance operator   $\mathcal{C}_{\phi}$ with eigenvalues $\{\lambda_i\}$.

Assume the posterior is denoted by $\mu_n^y$. $\beta_n$ is a rate of contraction if $\mathbb{Q}_{u_0,n}\mu_n^y\{u : \|u - u_0\|^2 \geq \beta_n^2\} \to 0$ as $n \to 0$. Opening the expression in $\{\phi_i\}$, we now have
\begin{eqnarray*}
\mathbb{Q}_{u_0,n}\mu_n^y\{u : \|u - u_0\|^2 \geq \beta_n^2\} &=& \mathbb{Q}_{u_0,n}\mu_n^y\{u : \sum_i(u_i - u_{0,i})^2 \geq \beta_n^2\}\\
&\geq& \mathbb{Q}_{u_0,n}\mu_n^y\{u : \sum_i(u_{2i - 1} - u_{0,2i - 1})^2 \geq \beta_n^2\}\\
\end{eqnarray*}
 Now consider the restricted model $$y = G_2(u) + \frac{1}{\sqrt{n}}\eta$$  with $u \in \mathcal{H'}_1$. $G_2 : \mathcal{H'}_1 \to \mathcal{H}_2$ 
 is a compact operator with eigenpair $\{\phi_{2i - 1},\rho_{\sigma(2i - 1)}\}$. $\mathcal{H'}_1$ is the closed subspace of $\mathcal{H}_1$ 
 spanned by $\{\phi_{2i - 1}\}$. The prior has covariance operator $\mathcal{C}$ with eigenpair $\{\phi_{2i - 1}, \lambda_{2i - 1}\}$. From 
 the previous calculation, it is clear that if $\{\beta_n\}$ is a contraction rate for the model, it is also a contraction rate for the restricted model. 
 Thus $\{\beta_n\}$ is slower than the minimax rate for the restricted model. We now note that $\theta_i > \rho_{\sigma(2i - 1)}$. It is intuitively 
 clear that the minimax rate can be made arbitrarily slow by making $\theta_i \to 0$ at a fast enough rate. Hence, the proof is complete.
\end{proof}

\appendix
\section{Appendix}
\label{sec:app}
 \begin{theorem}\label{cent-small-ball}
Assume we have a centered Gaussian measure $\nu_n$ on a Hilbert space $\mathcal{H}$ with eigenvalues of the covariance operator 
given by $\frac{k_i p_i^2}{R_n^2}$ with $p_i^2 > p_{i + 1}^2 > 0$ and $ 0 < C_1 < k_i < C_2$. 
 Then we have, $$\phi(\nu_n,0,\epsilon) \leq - \sum_{i + 1}^N (\log p_N - \log p_i)$$ whenever $N$ is such 
 that $$\sum_{i = N + 1}^{\infty} p_i^2 <  \frac{R_n^2 \epsilon^2}{2C_2} ,\,\,\,\,\,\,\,\, \text{ and } \,\,\,\,\,\,\,\, N < K_1 \left(\frac{\epsilon R_n}{p_N}\right)^2$$ for some $K_1 > 0.$ 
 \end{theorem}
\begin{proof} We start with 
\begin{equation*}
\phi(\nu_n,0,\epsilon) = - \log \left(\mathbb{P}\left(\sum v_i^2 \leq \epsilon^2\right)\right)
\end{equation*}
where $v_i$ are independent Gaussian random variables with mean $0$ and variance $\frac{k_i p_i^2}{R_n^2}$. Using independence, we have,
\begin{eqnarray}
\mathbb{P}\left(\sum v_i^2 \leq \epsilon^2\right) &\geq& \mathbb{P}\left(\sum_{i = 1}^{N} v_i^2 \leq \frac{\epsilon^2}{2}\right)\mathbb{P}\left(\sum_{i = N + 1}^{\infty} v_i^2 \leq \frac{\epsilon^2}{2}\right)
\end{eqnarray}
For the second term on the right hand side, we apply Chebyshev's inequality to obtain
\begin{eqnarray*}
  \mathbb{P}\left(\sum_{i = N + 1}^{\infty} v_i^2 \leq \frac{\epsilon^2}{2}\right) \geq   1 - \frac{2\sum_{i = N + 1}^{\infty}\mathbb{E}v_i^2}{\epsilon^2} \ge 1 - 2C_2\frac{\sum_{i = N + 1}^{\infty} p_i^2}{R_n^2 \epsilon^2}
\end{eqnarray*}
Thus the tail term is bounded away from $0$ by the assumption on $N$. For the first term right hand side, we proceed as in Lemma 6.2 in Belitser and Ghosal \cite{Belitser}
with centered Gaussian with variance $\{\frac{k_ip_i^2}{R_n^2}\}$ to get
\begin{eqnarray*}
\mathbb{P}\left(\sum_{i = 1}^{N} v_i^2 \leq \frac{\epsilon^2}{2}\right) \geq \left(\prod_{i = 1}^N{\frac{p_N}{p_i}}\right)\mathbb{P}\left(\sum_{i = 1}^{N}u_i^2 \leq \frac{(R_n \epsilon)^2}{2p_N^2}\right)
\end{eqnarray*}
Where $u_i$ are i.i.d standard normals. Applying Chebyshev's inequality again, and the second condition on $N$, we can show that $\mathbb{P}\left(\sum_{i = 1}^{N}u_i^2 \leq \frac{(R_n \epsilon)^2}{2p_N^2}\right)$
is bounded away from $0$, hence the theorem follows.
\end{proof}

Note that we did not directly use the non-centered version in Lemma 6.2 \cite{Belitser} because we do not know how the bias term $G(u_0)$ expands in term of the basis of $\nu_n$. 
\begin{corollary}\label{mild}
If $p_i = i^{-d}$, then we have $$\phi(\nu_n,0,\epsilon) \leq C_3 (\epsilon R_n)^{-\frac{2}{2d - 1}}.$$
\end{corollary}
\begin{proof}
The choice of $N$ can be taken to be such that $N \geq C_4(\epsilon R_n)^{-\frac{2}{2d - 1}}$. This gives $$\phi(\nu_n,0,\epsilon) \leq C_3 (\epsilon R_n)^{-\frac{2}{2d - 1}}.$$
\end{proof}

\begin{corollary}\label{severe}
If $p_i = i^s\exp(-i^{d})$, then we have $$\phi(\nu_n,0,\epsilon) \leq C_5 (- \log(\epsilon R_n))^{\frac{d + 1}{d}}.$$
\end{corollary}
\begin{proof}
The choice of $N$ can be taken to be such that $N \geq C_6(- \log (\epsilon R_n))^{\frac{1}{d}}$. This gives $$\phi(\nu_n,0,\epsilon) \leq C_5 (- \log(\epsilon R_n))^{\frac{d + 1}{d}}.$$
\end{proof}
\begin{theorem}\label{non-cent-small-ball-1}
Let $G$, $\mathcal{C}_n$ and $\nu_n$ be defined as in example discussed in Section \ref{sec:meyer}. Assume $u_0 \in \calH_1^{\gamma}(\mathcal{C}_n)$ and let 
$\sum_i \left(\frac{\left\langle u_0,\varphi_i\right\rangle}{i^{-\gamma}}\right)^2 = R$. Then we have, $$\mathcal{K}(\nu_n, G(u_0), \epsilon) \leq KR_n^2\left(\epsilon^{-\frac{1 + 2\delta - 2\gamma}{\alpha + \gamma}}\vee 1\right)$$ for some $K > 0$.
\end{theorem}

\begin{proof}
Assume $u_0 \in \calH_1^{\gamma}(\mathcal{C}_n)$ and let 
$\sum_i \left(\frac{\left\langle u_0,\varphi_i\right\rangle}{i^{-\gamma}}\right)^2 = R$. 
Let $u_0 = \sum_i u_{0,i}\varphi_i$, and $h = \sum_{i = 1}^{j_0}u_{0,i}G(\varphi_i)$ for some $j_0$.

For such an $h$, notice that
\[ \|h\|_{\nu_n}^2 = \left\langle (G\mathcal{C}_n G^T)^{-1/2}h,(G\mathcal{C}_n G^T)^{-1/2}h\right\rangle_2 = \left\langle (G\mathcal{C}_n G^T)^{-1}h,h\right\rangle_2  = \left\langle \mathcal{C}_n^{-1}G^{-1}h,G^{-1}h\right\rangle_2,\]
implying
\[ \|h\|_{\nu_n}^2 = \left\langle \mathcal{C}_n^{-1}\sum_{i = 1}^{j_0}u_{0,i}\varphi_i,\sum_{i = 1}^{j_0}u_{0,i}\varphi_i\right\rangle_1 = \sum_{i = 1}^{j_0} (R_n^2\lambda_i^{-1} i^{-2\gamma})u_{0,i}^2 i^{2 \gamma} 
\leq RR_n^2\left(\lambda_{j_0}^{-1} j_0^{-2\gamma}\vee 1\right).  \]
The last inequality is true since we know that the sequence $\{\lambda_{j} j^{2\gamma}\}$ is either increasing or decreasing monotonically.
Next, for any operator $K$ let us define $\left.K\right|_{S}$ as the operator $K$ restricted to the subset $S$. Writing $C_1^j$ as the linear span of 
$\{\varphi_{j + 1}, \varphi_{j + 2}....\}$, note that $C_1^j \subset (D_1^{j/4})^{\perp}$, therefore,
\[ \|G(u_0) - h\|_2^2 = \left\|G\sum_{i > j_0}u_{0,i} \varphi_i\right\|_2^2 \leq \left\|\left.G\right|_{C_1^{j_0}}\right\|^2 \left\|\sum_{i > j_0}u_{0,i} \varphi_i\right\|_1^2 
\leq  \left\|\left.G\right|_{(D_1^{j_0/4})^{\perp}}\right\|^2 j_0^{-2\gamma} \left(\sum_{i > j_0} u_{0,i}^2 i^{2\gamma}\right).\]
Hence, we have 
\[ \|G(u_0) - h\|_2^2 \leq R \rho^2_{j_0/4} j_0^{-2\gamma}. \]
We need $R \rho^2_{j_0/4} j_0^{-2\gamma} \leq \frac{\epsilon^2}{2}$. Hence, $j_0 \geq C \epsilon^{-\frac{1}{\alpha + \gamma}}$. Given that $\lambda_i = i^{-1 - 2\delta}$. This implies $\|h\|_{\nu_n}^2 \leq KR_n^2\left(\epsilon^{-\frac{1 + 2\delta - 2\gamma}{\alpha + \gamma}}\vee 1\right)$ for some $K > 0$. Since $\mathcal{K}(\nu_n, G(u_0), \epsilon) \leq \|h\|_{\nu_n}^2$, the result follows.
\end{proof}
\begin{theorem}\label{non-cent-small-ball-2}

Let $G$, $\mathcal{C}_n$ and $\nu_n$ be defined as in the example discussed in Section \ref{sec:G without SVD}. Assume $u_0 \in \calH_1^{\gamma}(\mathcal{C}_n)$ and let 
$\sum_i \left(\frac{\left\langle u_0,\varphi_i\right\rangle}{i^{-\gamma}}\right)^2 = R$. Further, let $\gamma' \equiv \frac{\gamma(1 - r)}{t}$. Then we have, $$\mathcal{K}(\nu_n, G(u_0), \epsilon) \leq KR_n^2\left(\epsilon^{-\frac{1 + 2\delta - 2\gamma'}{\alpha + \gamma'}}\vee 1\right)$$ for some $K > 0$.
\end{theorem}
\begin{proof}
Take $h = \sum_{i =1}^{j_0}u_{0,i}G^{1 - r}(\varphi_i)$ for some $j_0>0$.
Like in the previous theorem, we then estimate the Cameron-Martin norm of
$h$ under the measure $\nu_n$ and ensure that $h$ is close enough to $G(u_0)$. We also know that $j_0^{2\gamma}\rho_{j_0}^{lt}$ is either
increasing, or decreasing monotonically. Then, for some $K>0$
\begin{align*}
\|h\|^2_{\nu_n}  &= \left\|\left(G^{-r + 1}\mathcal{C}_nG^{-r +
1}\right)^{-1/2}h\right\|^2 \\
&= \|\mathcal{C}_n^{-1/2}G^{r-1}h\|^2 \\
&=  \sum_{j = 1}^{j_0} \frac{u_{0,j}^2}{\rho_j^{lt}}\\
&\leq K R R_n^2
\left(j_0^{-2\gamma + 1 + 2\delta}\vee 1\right) \\
&\leq K R R_n^2\left(j_0^{-2\gamma' + 1 + 2\delta}\vee 1\right)
\end{align*}
We now estimate $\|h - G^{1-r}u_0\|$. Note
that $\|G^{1-r}x\| = \left\|\left(\left(\mathcal{C}_n^{-1/l} -
K_2\right)^{-1}\right)^{\frac{1-r}{t}}x\right\| \leq \left\|\left(\mathcal{C}_n^{-1/l}
- K_2\right)^{-1}x\right\|^{\frac{1-r}{t}}$. The last inequality is true
because of Theorem \ref{thm:functional2}, and the fact that $t \geq (1 -
r)$. Next, writing
$M = \|\I + G^tK_2\| = \|(\I - \calC^{\frac1{l}}K_2)^{-1}\|$, we have
\begin{align*}
\|h - G^{1-r}u_0\| &= \left\|G^{1-r}\sum_{j = j_0}^{\infty}u_{0,j}\varphi_j\right\|\\
&\leq  \left\|\left(\mathcal{C}_n^{-1/l} - K_2\right)^{-1}\sum_{j =
j_0}^{\infty}u_{0,j}\varphi_j\right\|^{\frac{1-r}{t}}\\ &\leq 
(MR)^{\frac{1-r}{t}}j_0^{-\frac{\gamma(1 - r)}{t}}\rho_{j_0}^{1 - r}\\ &= Lj_0^{-\gamma' - \alpha}
\end{align*}
We need $Lj_0^{-\gamma' - \alpha} \leq \frac{\epsilon}{2}$. This implies $j_0 \geq C \epsilon^{\frac{-1}{\alpha + \gamma'}}$. Hence, $\|h\|^2_{\nu_n} \leq K_1 R_n^2\left(\epsilon^{-\frac{1 + 2\delta - 2\gamma'}{\alpha + \gamma'}}\vee 1\right)$. We thus have the same bound as in the previous theorem with $\gamma$ relaced by $\gamma'$.
\end{proof}
\begin{theorem}\label{non-cent-small-ball-3}

Let $G$, $\mathcal{C}_n$ and $\nu_n$ be defined as in the example discussed in Section \ref{severely ill posed}. Assume $u_0 \in \calH_1^{\gamma}(\mathcal{C}_n)$ and let 
$\sum_i \left(\frac{\left\langle u_0,\varphi_i\right\rangle}{i^{-\gamma}}\right)^2 = R$. Then we have, $$\mathcal{K}(\nu_n, G(u_0), \epsilon) \leq KR_n^2\left((-\log (\epsilon))^{\frac{1 + 2\delta - 2\gamma}{\beta}}\vee 1\right)$$ for some $K > 0$.
\end{theorem}
\begin{proof}
Take $h = \sum_{i =1}^{j_0}u_{0,i}\mathcal{C}^{-r}\exp (-\mathcal{C}^{-\beta'})(\varphi_i)$ for some $j_0>0$.
Again, we estimate the Cameron-Martin norm of
$h$ under the measure $\nu_n$ and ensure that $h$ is close enough to $G(u_0)$. 
\begin{align*}
\|h\|^2_{\nu_n}  &=  \sum_{j = 1}^{j_0} R_n^2 j^{1 + 2\delta}u_j^2\\
&\leq R R_n^2
\left(j_0^{-2\gamma + 1 + 2\delta}\vee 1\right) \\
\end{align*}
We now estimate $\|h - \mathcal{C}^{-r}\exp (-\mathcal{C}^{-\beta'})u_0\|$. Borrowing notations from Section \ref{sec:meyer}, we have
\begin{align*}
\|h - \mathcal{C}^{-r}\exp (-\mathcal{C}^{-\beta'})u_0\| &= \left\|\mathcal{C}^{-r}\exp (-\mathcal{C}^{-\beta'})\sum_{j = j_0}^{\infty}u_{0,j}\varphi_j\right\|\\
&\leq \left\|\mathcal{C}^{-r}\exp (-\mathcal{C}^{-\beta'})|_{C_1^{j_0}}\right\|\|\sum_{j = j_0}^{\infty}u_{0,j}\varphi_j\|\\
&\leq j_0^{(1 + 2\delta)r - 2\gamma}\exp(-j_0^{\beta})\\
\end{align*}
Thus, for $j_0 = (-\log (\epsilon))^{\frac{1}{\beta}} $, $\|h - \mathcal{C}^{-r}\exp (-\mathcal{C}^{-\beta'})u_0\| \leq \epsilon/2$. This gives $$\mathcal{K}(\nu_n, G(u_0), \epsilon) \leq R R_n^2
\left((-\log (\epsilon))^{\frac{-2\gamma + 1 + 2\delta}{\beta}}\vee 1\right)$$
\end{proof}
\bibliographystyle{abbrv}
\bibliography{madhuresh13}

\end{document}